\documentclass[11pt]{amsart}
\usepackage{amsmath,amsthm}
\usepackage{amssymb}
\usepackage{epstopdf, xypic}
\newtheorem{thm}{Theorem}[section]
\newtheorem{prop}[thm]{Proposition}
\newtheorem{Ass}[thm]{Assumptions}
\newtheorem{lem}[thm]{Lemma}
\newtheorem{cor}[thm]{Corollary}

\newtheorem{qst}[thm]{Question}

\theoremstyle{definition}

\newtheorem{eg}[thm]{Example}

\renewcommand{\phi}{\varphi}
\newcommand{\im}{\mathop{\mathrm{Im}}\nolimits}
\renewcommand{\ker}{\mathop{\mathrm{Ker}}\nolimits}
\newcommand{\coker}{\mathop{\mathrm{Coker}}\nolimits}

\newcommand{\iso}{\cong}
\renewcommand{\bar}{\protect\overline}

\renewcommand{\i}[1]{\mathfrak{#1}}
\newcommand{\p}{\i{p}}
\newcommand{\q}{\i{q}}

\newcommand{\y}{\mathbf y}

\newcommand{\rank}{\mathop{\mathrm{rank}}\nolimits}

\newcommand{\ass}{\mathop{\mathrm{Ass}}\nolimits}
\newcommand{\h}{\mathop{\mathrm{ht}}\nolimits}

\newcommand{\pd}{\mathop{\mathrm{pd}}\nolimits}
\newcommand{\reg}{\mathop{\mathrm{reg}}\nolimits}

\newcommand{\tor}{\mathop{\mathrm{Tor}}\nolimits}

\newcommand{\ext}{\mathop{\mathrm{Ext}}\nolimits}

\renewcommand{\*}{\bullet}

\begin{document}

%\begin{frontmatter}
\title{A Tight Bound on the Projective Dimension of Four Quadrics}
%\date{}                                           % Activate to display a given date or no date
\author{Craig Huneke}
\address{University of Virginia, Department of Mathematics, 141 Cabell Drive, Kerchof Hall, P.O. Box 400137, Charlottesville, VA 22904-4137}
\email{huneke@virginia.edu}
\author{Paolo Mantero}
\address{Department of Mathematical Sciences,
309 SCEN - 1, University of Arkansas - Fayetteville, AR 72701}
\email{pmantero@uark.edu}
\author{Jason McCullough}
\address{Department of Mathematics, Iowa State University, Ames, IA 50011}
\email{jmccullo@iastate.edu}
\author{Alexandra Seceleanu}
\address{University of Nebraska, Department of Mathematics, 203 Avery Hall, Lincoln, NE 68588}
\email{aseceleanu@unl.edu}

\begin{abstract}
Motivated by Stillman's question, we show that the projective dimension of an ideal generated by four quadric forms in a polynomial ring is at most $6$; moreover, this bound is tight.  We achieve this bound, in part, by giving a characterization of the low degree generators of ideals primary to height three primes of multiplicities one and two.    \end{abstract}

%\begin{keyword} projective dimension \sep free resolution \sep polynomial ring \sep quadric.
%\MSC[2010] Primary: 13D05 \sep Secondary:  14M07\sep 13C40 \sep 13D02.
%\end{keyword}

%\end{frontmatter}

\maketitle

\section{Introduction}

Let $S = k[x_1,\ldots,x_n]$ be  a polynomial ring over a field $k$, and let $I = (f_1,\ldots,f_N)$ be a homogeneous ideal of $S$.  Classical theorems of Hilbert or Auslander and Buchsbaum give upper bounds on the projective dimension of $S/I$ (or, equivalently, of $I$) in terms of $n$, the number of variables of $S$.  Motivated by computational efficiency issues, Stillman \cite{PS} posed the following question:
\begin{qst}[Stillman {\cite[Problem 3.14]{PS}}]\label{Stillman} Is there a bound on $\pd(S/I)$ depending only on $d_1,\ldots,d_N$, and $N$, where $d_i = \deg(f_i)$?
\end{qst}
Ananyan and Hochster \cite{AH2} recently proved an affirmative answer to Stillman's Question in full generality.  They prove that for any homogeneous ideal $I$ generated by $N$ forms of degree at most $d$ either the generators form a regular sequence, implying that $N$ is an upper bound for the projective dimension, or else the generators of $I$ are contained in an ideal generated by a bounded number (in terms of $N$ and $d$) of forms of strictly smaller degree.  A delicate inductive proof shows the existence of a bound, but even in the case when $I$ is generated by quadrics, the bounds they produce are very large and far from optimal.  Tighter bounds in more specific cases have been given in the following situations:
\begin{enumerate}
\item In a previous paper \cite{AH}, Ananyan and Hochster showed that if $I$ is generated by $N$ quadric forms, then $\pd(S/I)$ has an upper bound asymptotic to $2N^{2N}$.  They also give more general result for non-homogeneous ideals generated in degree at most $2$.  In \cite[Theorem 4.2]{AH2} they claim a new upper bound of $2^{N+1}(N-2) + 4$. 
\item When $I$ is minimally generated by $N$ quadrics and ${\rm ht}(I) = 2$, we previously showed \cite{HMMS} that $\pd(S/I) \le 2N - 2$.  (See Theorem~\ref{ht2}.)  Moreover, for all $N$ we construct examples where equality is achieved.  Thus the bound here is optimal.
\item When $I$ is minimally generated by $3$ cubics, Engheta \cite{En3} showed that $\pd(S/I) \le 36$, while the largest known example satisfies $\pd(S/I) = 5$.  
\item When $\mathrm{k} = p > 0$ and $S/I$ is an $F$-pure ring, De Stefani and N\'u\~nez-Betancourt \cite{DN} showed that $\pd(S/I) \le \mu(I)$, where $\mu(I)$ denotes the minimal number of generators of $I$.
\end{enumerate}

It remains open as to what the optimal bounds for projective dimension are.  The general optimal bound, however, must be rather large.  
Examples of Beder et. al. \cite{BMNSSS} show that the projective dimension of ideals generated by $3$ degree-$d$ forms can grow exponentially with respect to $d$.  

Further motivating Question~\ref{Stillman}, Caviglia showed it was equivalent to the following analogous question for Castelnuovo-Mumford regularity. (See \cite[Theorem 2.4]{MS}.) 
\begin{qst}[{Stillman \cite[Problem 3.15]{PS}}]\label{Stillman2}
Is there a bound on $\reg(S/I)$ depending only on $d_1,\ldots,d_N$, and $N$, where $d_i = \deg(f_i)$?
\end{qst}

While $(1)$  above shows that $\pd(S/I)$ is bounded for any ideal generated by $N$ quadrics, these bounds are exponential in $N$.  It is clear already in the case of an ideal generated by $3$ quadrics that these bounds are far from optimal (cf. \cite[Proposition 24]{MS}).  In this paper we give a tight upper bound on the projective dimension of an ideal generated by $4$ quadrics.  Specifically, we prove
\begin{thm}
Let $S$ by a polynomial ring over a field $k$, and let $I = (q_1,q_2,q_3,q_4)$ be an ideal of $S$ generated by $4$ homogeneous polynomials of degree $2$.
Then ${\rm pd}(S/I)\leq 6$.
\end{thm}
Examples from \cite{Mc} of ideals $I$ generated by $4$ quadrics with $\pd(S/I) = 6$ show that the above bound is optimal.  Our proof follows a technique similar to Engheta's in \cite{En} and \cite{En3} in dividing into cases by height and multiplicity.  We notably rely on the tight bound for the height $2$ case proved in \cite{HMMS} and the height $3$ multiplicity $6$ case proved in \cite{HMMS3}.  The remaining height $3$ cases constitute the bulk of this paper and require a variety of approaches.  

Previously the authors had posted a paper to the arXiv proving an upper bound of $9$; however, we suspected that $6$ should be the tight upper bound.  We finally reduced all of the remaining cases to $6$ producing this paper.  This was done at the expense of the length and complexity of the paper.  However, we do not believe these arguments can be significantly shortened without weakening the upper bound and, along with Theorem~\ref{ht2}, this paper represents the only other nontrivial known tight bound for Stillman's Question.  A table summarizing the many cases is given in Section~\ref{main}.

The rest of the paper is organized as follows:
in Section~\ref{back}, we collect many of the results that we cite and fix notation; in Section~\ref{basic}, we prove some results that are used several times throughout the paper; in Sections~\ref{primary}, \ref{22struct}, we characterize certain height three primary ideals that can occur as the unmixed part of an ideal of four quadrics or a direct link of such an ideal; in Sections~\ref{e2}, \ref{e35}, \ref{e4}, we handle the cases where our ideal has multiplicity $2$, $3$, $4$ or $5$.  Section~\ref{leftover}, handles the cases where the ideal has multiplicity $4$ or $5$ and we know that any complete intersection contained in our ideal is extended from a $6$ variable polynomial ring.
Finally, we collect all of the previous results to prove our main result in Section~\ref{main}.  \ref{append} contains a list of the many primary ideals needed for Sections~\ref{primary} and \ref{22struct}.  It may be skipped without affecting the logical flow of the paper, but it may also be of independent interest.

\section{Background and Notation}\label{back}

In this section we set the notation that will be used throughout this manuscript and collect several results that will be used in the following sections.

\subsection{Unmixed ideals and multiplicity}

We will frequently use  the following well-known \textit{associativity formula} (also sometimes referred to as the \textit{linearity formula} or \textit{additivity and reduction formula}) to compute the multiplicity of an ideal.
\begin{prop}[cf. {\cite[Theorem 11.2.4]{SH}}]\label{Ass}
If $J$ is an ideal of $S$, then $$e(S/J)=\sum_{\substack{\rm{primes\,\, } \mathfrak{p} \supseteq J\\ \rm{ht}(\mathfrak{p}) = \rm{ht}(J)}}e(S/\mathfrak{p})\lambda(S_\mathfrak{p}/J_\mathfrak{p}).$$
\end{prop}
Here $e(R)$ denotes the multiplicity of a graded ring $R$, and $\lambda(M)$ denotes the length of an $S$-module $M$.  

Recall that an ideal $J$ of height $h$  is {\it unmixed} if ${\rm ht}(\mathfrak{p}) = h$ for every $\p\in {\rm Ass}(S/J)$. Also, the {\it unmixed part} of $J$, denoted $J^{un}$, is the intersection of all the components of $J$ of minimum height. One has $J\subseteq J^{un}$, and from the associativity formula it follows that $e(S/J) = e(S/J^{un})$.  We follow \cite{En} in using the following notation: if $J$ is an unmixed ideal, we say it is of type
\[\langle e_1,\ldots,e_m; \lambda_1,\ldots,\lambda_m\rangle\]
if $J$ has $m$ associated prime ideals $\p_1,\ldots, \p_m$ of equal height with $e_i = e(S/\p_i)$ and $\lambda_i = \lambda(S_{\p_i}/J_{\p_i}) $ for all $i$.  It follows that $e(S/J) = \sum_{i = 1}^m e_i \lambda_i$.

The following simple fact also follows from the associativity formula:
\begin{lem}[{cf. \cite[Lemma 8]{En}}]\label{unmixedequal}
Let $J \subseteq S$ be an unmixed ideal.  If $I \subseteq S$ is an ideal containing $J$ such that $\textrm{ht}(I) = \textrm{ht}(J)$ and $e(S/I) = e(S/J)$, then $J = I$.
\end{lem}

Engheta's finite characterization of height $2$ unmixed ideals of multiplicity $2$ gives a bound on projective dimension in this case.  
\begin{prop}[{\cite[Proposition 11]{En}}]\label{h2e2} Let $J$ be an unmixed ideal in a polynomial ring $S$ over an algebraically closed field.  Suppose ${\rm ht}(J) = 2$ and $e(S/J) = 2$.  Then $\pd(S/J) \le 3$.

In particular, if $J$ is primary to $(x,y)$ for independent linear forms $x, y$, then
\begin{enumerate}
\item $J = (x, y^2)$ or
\item $J = (x^2, xy, y^2 , ax + by)$, where ${\rm ht}(x, y, a, b) = 4$.
\end{enumerate}
\end{prop}

\noindent In \cite{HMMS2} it was shown that such a complete characterization is very special and cannot be extended to higher heights or multiplicities.  In fact, for any multiplicity $e \ge 2$ and any height $h \ge 2$ with $(e,h) \neq (2,2)$ there are infinitely many nonisomorphic primary ideals of multiplicity $e$ and height $h$.  Moreover, they can be chosen to have arbitrarily large projective dimension.  However, in Sections~\ref{primary} and \ref{22struct} we give a characterization of ideals primary to a height three linear prime ideal or to a height three multiplicity two prime that is sufficient for the bounds we want.

\subsection{Linkage}

Two ideals $J$ and $K$ in a regular ring $R$ are said to be {\it linked}, denoted $J\sim K$, if there exists a regular sequence $\underline{\alpha}=\alpha_1,\ldots,\alpha_g$ such that $K=(\underline{\alpha}):J$ and $J=(\underline{\alpha}):K$.  Notice that the definition forces $J$ and $K$ to be unmixed, and $(\underline{\alpha}) \subseteq J\cap K$.
Linkage has been studied since the nineteenth century, although its first modern treatment appeared in the ground-breaking paper by Peskine and Szpiro \cite{PSz}. We refer the interested reader to \cite{M} and \cite{HU} and their references.

We will need a few results regarding linkage.  

\begin{thm}$(${\rm Peskine-Szpiro}, \cite{PSz}$)$\label{linkage}
Let $J$ be an unmixed ideal of $S$ of height $g$. Let $\underline{\alpha}=\alpha_1,\ldots,\alpha_g$ be a regular sequence in $J$, and set $K=(\underline{\alpha}):J$. Then, one has
\begin{enumerate}
\item  $J=(\underline{\alpha}):K$, that is, $J\sim K$ via $\underline{\alpha};$
\item $S/J$ is Cohen-Macaulay if and only if $S/K$ is Cohen-Macaulay$;$
\item $e(S/J)+e(S/K)=e(S/(\underline{\alpha})).$
\end{enumerate}
\end{thm}
It is easily checked that if $\underline{\alpha}$ is a regular sequence of maximal length in an ideal $J$, then $(\underline{\alpha}):J=(\underline{\alpha}):J^{un}$, that is, $J^{un}$ is linked to $(\underline{\alpha}):J$. We will use this fact several times.

The following facts are well-known. 

\begin{lem}[{cf. \cite[Lemma 2.6]{EnTh}}] \label{doublelink}
Suppose $J$ and $L$ are unmixed ideals both linked to the same ideal.  Then $\pd(S/J) = \pd(S/L)$.
\end{lem}

\begin{lem}[{cf. \cite[Theorem 7]{En}}]\label{ses}
Let $J$ be an almost complete intersection ideal of $S$. If $K$ is any ideal linked to $J^{un}$, then one has $${\rm pd}(S/J)\leq {\rm pd}(S/K)+1.$$
\end{lem}

\subsection{Height three primes of small multiplicity}

Here we recall some of the structure theorems regarding prime ideals of small height and multiplicity.  For most of this section we need the field $k$ to be algebraically closed. The reduction to this case occurs in the proof of Theorem~\ref{pd6}.  Also, recall that a homogeneous ideal $J$ is called {\it degenerate} if $J$ contains at least one linear form, otherwise $J$ is said to be {\it non-degenerate}.  

The following result is an ideal-theoretic version of a classical result of Samuel and Nagata. 

\begin{thm}[ Samuel \cite{Sa} , Nagata {\cite[Theorem~40.6]{Na}}]\label{SN}
Let $J$ be a homogeneous, unmixed ideal of S. If $e(S/J) = 1$, then $J$ is generated by $\h(J)$ linear forms.
\end{thm}

The next result we need is a classical simple lower bound for the multiplicity of non-degenerate prime ideals.

\begin{prop}[{\cite[Corollary~18.12]{Ha}}]\label{non-deg}
Let $\mathfrak{p}$ be a homogeneous prime ideal of $S$. If $\mathfrak{p}$ is a non-degenerate prime ideal, then $e(S/\mathfrak{p})\geq{\rm ht}(\mathfrak{p})+1$.
\end{prop}

The following result is a consequence of Proposition~\ref{non-deg}.

\begin{cor}\label{e=2}
Let $\mathfrak{p}$ be a homogeneous prime ideal of $S$ of height three. If $e(S/\mathfrak{p})=2$, then there exist linear forms $x,y$ and a quadric $q$ such that $\mathfrak{p}=(x,y,q)$.
\end{cor}

The next two results are classification theorems of varieties of small degree over an algebraically closed field. We give an equation-oriented version of these results that best fits our purpose. The first of these results follows from a classical classification theorem for varieties of multiplicity $3$.  (See \cite{S}.)  By Proposition~\ref{non-deg} all such primes are degenerate.

\begin{thm}[{\rm Anonymous \cite{X}, Swinnerton-Dyer \cite[Theorem~3]{S}}]\label{e=3}
Let $\mathfrak{p}$ be a homogeneous prime ideal of $S$ of height three. If $e(S/\mathfrak{p})=3$, then $\mathfrak{p}$ is  of one of the following types:
\begin{enumerate}
\item $\mathfrak{p}=(x,y,c)$ where $x$ and $y$ are linear forms and $c$ is a cubic form;
\item $\mathfrak{p}=(x) + I_2(\boldsymbol{M})$, where $\boldsymbol{M}$ is a $2 \times 3$ matrix of linear forms and $x$ is a linear form that is necessarily a nonzerodivisor on $I_2(\boldsymbol{M})$.  
\end{enumerate}
\end{thm}

The next result is a classification theorem for varieties of multiplicity $4$.  The non-degenerate case was done by Swinnerton-Dyer \cite[Theorem 1]{S}; the degenerate cases are derived from work of Brodmann-Schenzel \cite[Theorem~2.1]{BS}). \begin{thm}\label{e=4}
Let $p$ a homogeneous prime ideal of $S$ of height $3$. If $e(S/\mathfrak{p})=4$, then 
$\mathfrak{p}$ is of one of the following types:
\begin{enumerate}
\item $\mathfrak{p}=(x,y,r)$, where $x,y$ are linear forms and $r$ is a quartic;
\item $\mathfrak{p}=(x,q,q')$, where $x$ is a linear form and $q$ and $q'$ are quadrics;
\item $\mathfrak{p} = I_2(\boldsymbol{M})$, where $\boldsymbol{M}$ is a $2 \times 4$ matrix of linear forms or a $3 \times 3$ symmetric matrix of linear forms;
\item either $\mathfrak{p}$ is generated by a linear form, a quadric and three cubics, or $\mathfrak{p}$ is generated by a linear form and $7$ cubics. \end{enumerate}
\end{thm}

\subsection{Related results}

Finally we recall two results of the authors' needed for the main result:

\begin{thm}[{\cite[Theorem 3.5]{HMMS}}]\label{ht2} For any ideal $I$ of height two generated by $n$ quadrics in a polynomial ring $S$, one has $\pd(S/I) \le 2n - 2$.  Moreover, this bound is tight.
\end{thm}
In particular, an ideal generated by $4$ independent quadrics of height two satisfies $\pd(S/I) \le 6$.  

The following is a generalization of a result of Engheta.

\begin{thm}[{\cite[Corollary 2.8]{HMMS3}}, {\cite[Theorem 1]{En2}}]\label{e6} Let $I$ be an almost complete intersection of height $g > 1$ generated by quadrics.  Then $e(S/I) \le 2^g - g + 1$.  Moreover, if $e(S/I) = 2^g - g + 1$, then $S/I$ is Cohen-Macaulay.
\end{thm}
Engheta proved an upper bound in the case of any graded almost complete intersection.  The Cohen-Macaulayness and a more general bound are given in \cite{HMMS3}.
In particular, this result shows that an ideal $I$ of height $3$ generated by $4$ quadrics has multiplicity at most $6$, and moreover, if $I$ has multiplicity exactly $6$, it is Cohen-Macaulay and hence $\pd(S/I) = 3$.

\section{Basic Results}\label{basic}

In this section we collect some basic results and notation.  If $I$ is an ideal generated by $4$ quadric forms in a polynomial ring $S$, then $\textrm{ht}(I) \le 4$.  If $\textrm{ht}(I) = 1$ or $4$, then $\pd(S/I) = 4$.  By Theorem~\ref{ht2}, it suffices to consider the height three case.  

\begin{Ass}\label{Not} Unless otherwise specified, we use the following notation in the remainder of the paper$:$
\begin{itemize}
\item $S$ is a polynomial ring over an algebraically closed field $k,$
\item $I = (q_1, q_2, q_3, q_4)$ is an $S$-ideal of height $3,$
\item $q_1, q_2, q_3, q_4$ are independent, homogeneous polynomials of degree $2,$
\item ${\rm ht}(q_1, q_2, q_3)=3$.  
\item $L = (q_1, q_2, q_3):I = (q_1, q_2, q_3):I^{un}$.
\end{itemize} 
\end{Ass}

\begin{lem}\label{iuncm} If $I^{un}$ is Cohen-Macaulay, then $\pd(S/I) \le 4$.
\end{lem}

\begin{proof} By Proposition~\ref{linkage}, $(q_1,q_2,q_3):I = (q_1,q_2,q_3):I^{un}$ is Cohen-Macaulay.  By Lemma~\ref{ses}, $\pd(S/I) \le 4$.
\end{proof}

\begin{prop}\label{e=1}
If $e(S/I)=1$, then ${\rm pd}(S/I)\leq 4$.
\end{prop}

\begin{proof}
By Theorem~\ref{SN}, $I^{un}$ is Cohen-Macaulay, and by Lemma~\ref{iuncm}, ${\rm pd}(S/I)\leq 4$.
\end{proof}

\begin{lem}\label{LinForm}
If $I^{un}$ or $L = (q_1,q_2,q_3):I$ contains a linear form, then ${\rm pd}(S/I)\leq 5$.
\end{lem}

\begin{proof}
First suppose $x$ is a linear form contained in $I^{un}$. Since we have the containments $I\subseteq (I,x)\subseteq I^{un}$, it follows that ${\rm ht}\,(I,x)= 3$. Hence, after possibly relabeling the quadrics $q_1,\ldots,q_4$, we may assume $x,q_1,q_2$ form a regular sequence. Let $L':=(x,q_1,q_2):I=(x,q_1,q_2):I^{un}$.  We have now reduced to the case when an ideal $J$ (either $L'$ or $L$) directly linked to $I^{un}$ contains a complete intersection of one linear form and two quadric forms.

Since $J$ contains a complete intersection of one linear and two quadric forms, say $x, q_1, q_2$, $e(S/J) \le 4$.  If $e(S/J) = 4$, then $J = (x, q_1, q_2)$.  In particular, $J$ is Cohen-Macaulay and ${\rm pd}(S/I) \le 4$ by Lemma~\ref{ses}.  If $e(S/J) = 1$, then $J$ is Cohen-Macaulay by Theorem~\ref{SN} and again ${\rm pd}(S/I) \le 4$.  If $e(S/J) = 2$, then we can write $J = (x) + J^{\prime}$, where $J^{\prime}$ is unmixed, ${\rm ht}(J^{\prime}) = 2$ and $e(S/J^{\prime}) = 2$ such that $x$ is regular on $J^{\prime}$.  By Proposition~\ref{h2e2}, ${\rm pd}(S/J^{\prime}) \leq 3$.  This yields ${\rm pd}(S/J)\leq 4$ (because $x$ is regular on $S/J^{\prime}$) which, by Jemma \ref{ses}, gives ${\rm pd}(S/I)\leq 5$.  Finally if $e(S/J) = 3$, consider $K = (x,q_1,q_2):J$.  By Theorem~\ref{linkage} we have $e(S/K) = e(S/(x,q_1,q_2)) - e(S/J) = 4-3 = 1$.  So $K$ is Cohen-Macaulay by Theorem~\ref{SN}.  By Lemma~\ref{doublelink}, ${\rm pd}(S/I^{un}) = 3$, and by Lemma~\ref{iuncm}, ${\rm pd}(S/I) \le 4$.
\end{proof}

\begin{lem}\label{inters} Suppose $I \subseteq K \cap J$, where $K = K' + (q)$, ${\rm ht}(K) = {\rm ht}(J) = 3$, ${\rm ht}(K') = 2$ and $q$ is a quadric.  Then there exists a quadric $q^{\prime}\in K\cap J$ such that $K=K' + (q^{\prime})$.
\end{lem}

\begin{proof}
For each $i=1,\ldots,4$, write $q_i=f_i+\alpha_iq$, where $f_i$ is an element of $K^{\prime}$ and $\alpha_i \in k$.
If $\alpha_i = 0$ for all $i$, one has $I\subseteq K^{\prime}$, which implies ${\rm ht}\,K^{\prime}\leq 2$ and gives a contradiction. Hence, one may assume $\alpha_i\neq 0$ for some $i$. Now, take $q^{\prime}=q_i$.
\end{proof}

\begin{lem}\label{sum} Suppose $I_1, I_2$ are ideals.  Then 
\[\pd(S/(I_1 \cap I_2)) \le \max\{\pd(S/I_1), \pd(S/I_2), \pd(S/(I_1+I_2)) - 1\}.\]
\end{lem}

\begin{proof} This follows from the long exact sequence of $\tor_i^S(-,k)$ applied to the short exact sequence
\[0 \to S/(I_1 \cap I_2) \to S/I_1 \oplus S/I_2 \to S/(I_1 + I_2) \to 0.\]
\end{proof}

\begin{prop}\label{messy} Let $I$ be an almost complete intersection of height $h$.  Let $F_\*$ be the minimal free resolution of the unmixed part $I^{un}$ of $I$.  Let $\partial_i$ denote the $i^{\text th}$ differential in $F_\*$.  Then
\[\pd(S/I) \le \max\{h + 2, \pd(\coker(\partial_{h+1}^*))\}.\]
\end{prop}

\begin{proof}  Set $I = (f_1,\ldots, f_h, f_{h+1})$, where $C = (f_1,\ldots, f_h)$ is a complete intersection.  Then
\begin{align*}
\pd(S/I) &\le \max\{h,\pd(S/(C:f_{h+1}))+1\}\\
&\le \max\{h+1, \pd((C:f_{h+1})/C) + 2\}\\
&\le \max\{h + 2, \pd(\ker(\partial_{h+1}^*)) + 2\}\\
&= \max\{h + 2, \pd(\coker(\partial_{h+1}^*))\}.
\end{align*}
The first inequality follows from the the short exact sequence
\[0 \to S/(C:f_{h+1}) \to S/C \to S/I \to 0.\]
The second inequality follows from the short exact sequence
\[0 \to (C:f_{h+1})/C \to S/C \to S/(C:f_{h+1}) \to 0.\]
The third inequality follows from the isomorphism $(C:f_{h+1})/C \iso \ext_S^h(S/I^{un},S)$ (see \cite[Lemma 3.1]{HMMS2}) and the short exact sequence
\[0 \to \im(\partial_h^*) \to \ker(\partial_{h+1}^*) \to \ext_S^h(S/I^{un},S) \to 0.\]
By \cite[Lemma 3.3]{HMMS2}, $\pd(\im(\partial_h^*)) = h-1$, which implies that $\pd(\ext_S^h(S/I^{un},S)) = \pd(\ker(\partial_{h+1}^*)) $.  Finally, $\pd(\coker(\partial_{h+1}^*)) = \pd(\ker(\partial_{h+1}^*)) + 2$ gives the last equality.
\end{proof}

We make use of a slight generalization of an observation of Ananyan-Hochster:

\begin{lem}[{\cite[Lemma 3.3]{AH}}]\label{regseq}
Let $f_1,\ldots,f_t$ be a regular sequence of forms in $S$ and set $A = k[f_1,\ldots,f_t]$. Then for any ideal $I$ of $S$ extended from $A$ one has  $\pd(S/I) \le t$. More generally, for any finite $S$-module $M$ presented by a matrix with entries in $A$, one has $\pd(M) \le t$.  In particular, any ideal $I$ whose generators can be written in terms of at most $t$ variables satisfies $\pd(S/I) \le t$.
\end{lem}

In general $\pd(S/I)$ may be arbitrarily larger than $\pd(S/I^{un})$.  However, for almost complete intersection ideals whose unmixed part is extended from a smaller polynomial ring, we get the following useful corollary to Proposition~\ref{messy}.

\begin{cor}\label{iungens} Let $I = (g_1,\ldots,g_h, g_{h+1})$ be an almost complete intersection of height $h$.  Suppose $I^{un}$ is extended from $A = K[f_1,\ldots,f_t]$, where $f_1,\ldots,f_t$ are a regular sequence of forms.  Then $\pd(S/I) \le \max\{h+ 2, t\}$.
\end{cor}

\begin{proof}  Let $(F_i,\partial_i)$ be the minimal free resolution of $I^{un} \cap A$ over $A$.  Then the entries of $\partial_{h+1}$ lie in $A$.  Hence
\begin{align*}
\pd_S(S/I) &\le \max\{h+2, \pd_S(\coker(\partial_{h+1}^* \otimes S))\}\\
&= \max\{h+2, \pd_A(\coker(\partial_{h+1}^*))\}\\
&\le \max\{h+2, t\}.
\end{align*}
The first inequality follows from Proposition~\ref{messy}.  The second line follows from the faithful flatness of $S$ over $A$.  The last inequality follows from Lemma~\ref{regseq}.
\end{proof}

\section{$(x,y,z)$-primary ideals}\label{primary}

In this section we give a characterization of the degree $2$ component of an ideal $J$ primary to a height $3$ linear prime $\p = (x,y,z)$ with $e(S/J) = 2, 3$ or $4$.  As shown in \cite{HMMS2}, there are infinitely many distinct such ideals, but we show here that there are only a finite number of possibilities for the low degree generators of such ideals.  First we need a lemma regarding matrices of linear forms.  

For a matrix $\mathbf{M}$, we set $I_k(\mathbf{M})$ to be the ideal generated by the $k \times k$ minors of $M$.  Recall that $I_k(\mathbf{M})$ is unchanged by $k$-linear row and column operations.  We say that $\mathbf{M}$ has a generalized zero if after $k$-linear row and column operations $\mathbf{M}$ has a zero entry.  A matrix which has no generalized zeros is called \textit{$1$-generic}.  We refer the reader to \cite{Eisenbud2} for background on $1$-generic matricies.

\begin{lem}\label{matrixlemma}
Suppose $\mathbf{M}$ is a $2 \times 3$ matrix of linear forms in $S$ such that $\h(I_2(\mathbf{M})) = 1$.  Then, after $k$-linear row and column operations, $\mathbf{M}$ has one of the following forms:
\begin{enumerate}
\item $\mathbf{M} = \begin{pmatrix} a&0&0\\d&e&f\end{pmatrix}$, where $a \neq 0$ and $\h(e,f) = 2$;
\item $\mathbf{M} = \begin{pmatrix} a&b&0\\d&e&0\end{pmatrix}$, where $ae-bd \neq 0$;
\item $\mathbf{M} = \begin{pmatrix} a&b&0\\d&0&b\end{pmatrix}$, where $\h(a,b,d) = 3$.
\end{enumerate}
\end{lem}

\begin{proof}
Since $\h(I_2(\mathbf{M})) = 1$, $\mathbf{M}$ is not $1$-generic \cite[Theorem 6.4]{Eisenbud2}.  Hence $\mathbf{M}$ has a generalized zero and we may assume that it has the form
\[\mathbf{M} = \begin{pmatrix} a&b&0\\d&e&f\end{pmatrix}.\]
If $\h(a,b) \le 1$, then after a column operation, $\mathbf{M}$ has the form of case (1) above.  So we may assume that $\h(a,b) = 2$.  Suppose $\h(d,e,f) = 3$.  Then since $I_2(\mathbf{M}) = (af, bf, ae-bd)$ has height $1$, we have $ae - bd = xf$ for some $x \in S_1$.   Since $d,e,f$ form a regular sequence, we must have $a \in (d,f)$.  If $a \in (f)$, then it follows that $b \in (f)$ as well, contradicting that ${\rm ht}(a,b) = 2$.  Hence, after a row operation, we may now assume that $\h(d,e,f) \le 2$.  In this case, we can do a column operation so that $\mathbf{M}$ now has one of the following forms:
\[\mathbf{M} = \begin{pmatrix} a&b&0\\d&0&f\end{pmatrix} \qquad \text{or} \qquad \mathbf{M} = \begin{pmatrix} a&b&0\\d&e&0\end{pmatrix}.\]
We are done in the latter case.  In the former case, note that $I_2(\mathbf{M}) = (af, bf, bd) = (a,b) \cap (b,f) \cap (d,f)$.  Since $\h(a,b) = 2$, we must have $\h(b,f) \le 1$ or $\h(d,f) \le 1$.  If $\h(b,f) \le 1$, we are in case (3).  If $\h(d,f) \le 1$, we can perform row and column operations so that $\mathbf{M}$ has the form of case (1) or case (2).

Finally note that the conditions on the entries of the matrices distinguish the cases by the minimal number of generators of $I_2(\mathbf{M})$.
\end{proof}

\begin{lem}\label{matrixlemma2}
Suppose $\mathbf{M}$ is a $2 \times 3$ matrix of linear forms in $S$ such that $\h(I_2(\mathbf{M})) = 2$ and $\h(I_1(\mathbf{M})) = 2$.  Then, after $k$-linear row and column operations, $\mathbf{M}$ has the following form:
\[\mathbf{M} = \begin{pmatrix} a&b&0\\0&a&b\end{pmatrix}, \text{where $a,b$ are distinct linear forms.}\]
\end{lem}

\begin{proof} Since $\h(I_2(\mathbf{M})) = 2$, the entries in the first row must generate a height $2$ ideal.  After column operations we may assume that the top-right entry is $0$.  The bottom-right entry then cannot be $0$, since $\h(I_2(\mathbf{M})) = 2$.  After further column operations we may assume $\mathbf{M}$ has the form
\[\mathbf{M} =  \begin{pmatrix} a&b&0\\\alpha a + \beta b& \gamma a + \delta b&b \end{pmatrix}.\]
After more column operations we may force $\beta = \delta = 0$.  After a row operation we can force $\alpha = 0$.  Since $\h(I_2(\mathbf{M})) = 2$, $\gamma \neq 0$.  After a column operation, $\mathbf{M}$ has the desired form.
\end{proof}

\begin{prop}\label{1;2} Let $J$ be $\p$-primary with $e(S/J) = 2$.  Then one of the following holds:
\begin{enumerate}
\item $J = (x,y,z^2)$
\item $J = (x,y^2,yz,z^2,ay+bz)$, where $\h(x,y,z,a,b) = 5$
\item $J = (x,y,z)^2 + (ax+by+cz,dx+ey+fz)$, where $a,b,c,d,e,f \in S_1$, $\h(x,y,z,I_2(\mathbf{M})) \ge 5$, $\h(x,y,z,a,b,c,d,e,f) \ge 6$, and $\mathbf{M}$ is the matrix $\begin{pmatrix}a&b&c\\d&e&f\end{pmatrix}$.
\item All quadrics in $J$ can be written in terms of at most $6$ variables.  
\end{enumerate}
In particular, $\pd(S/J) \le 4$ for cases (1), (2) and (3).
\end{prop}

\begin{proof} Since $e(S/J) = 2$, we have $(x,y,z)^2 \subseteq J \subseteq (x,y,z)$.  If $J$ contains a linear form, which we may take to be $x$, then $J = (x) + J^{\prime}$, where $J^{\prime}$ is $(y,z)$-primary and $e(S/J^{\prime}) = 2$.  By Lemma~\ref{h2e2}, either $J = (x,y,z^2)$ or $J = (x, y^2, yz, z^2, ay + bz)$, where $\h(x,y,z,a,b) = 5$.  

Suppose $J$ does not contain a linear form.  Since $e(S/(x,y,z)^2) = 4$, $J$ must contain at least $2$ more generators of the form $ax + by + cz, dx + ey + fz$ not contained in $(x,y,z)^2$.  We may assume that both such generators are linearly independent and degree $2$ or else we are in case (4).  By Lemma~\ref{primary21}, the ideal $(x,y,z)^2 + (ax + by + cz, dx + ey + fz)$ is unmixed and has multiplicity $2$, and hence equal to $J$, if $\h(x,y,z,I_2) \ge 5$, where $I_2$ denotes the $2\times 2$ minors of the matrix $\mathbf{M} = \begin{pmatrix}a&b&c\\d&e&f\end{pmatrix}$.  So we may assume that $\h(I_2) = 1$ modulo $(x,y,z)$.  Hence, after a linear change of variables and a linear combination of generators, we may assume that $\mathbf{M}$ has one of the three forms in Lemma~\ref{matrixlemma} modulo $(x,y,z)$.  In case (1), $ax \in J$.  Since $a \notin \p$ and since $J$ is $\p$-primary, $x \in J$ contradicting that $J$ does not contain a linear form.  In case (2), we see that $(ae-db)y \in J$  and $ae-db \notin \p$ and hence $y \in J$, also a contradiction.  In the remaining case (3), we see that
\[b(dy - az) =  d(ax+by)-a(dx + bz) \in J.\]
Since $b \notin \p$, we have $dy - az \in J$ and $\h(x,y,z,a,b,d) = 6$.  By Lemma~\ref{primary22} the ideal $(x,y,z)^2 + (ax+by, dx+bz, dy - az)$ is unmixed of multiplicity $2$ and equal to $J$.  All quadrics in this ideal can be written in terms of the $6$ variables $x,y,z,a,b,d$ leaving us in case (4). This completes the proof.\end{proof}

\begin{prop}\label{1;3} Let $J$ be $\p$-primary with $e(S/J) = 3$.  Then one of the following holds:
\begin{enumerate}
\item $J = (x,y^2,yz,z^2)$. 
\item $J = (x,y^2,yz,z^3,ay+z^2)$, where $a \in S_1$ and $\h(x,y,z,a) = 4$. 
\item $J = (x^2,xy,xz,y^2,yz,ax+by+z^2,cx+dy)$, where $a,b,c,d \in S_1$ and $\h(x,y,z,c,d) = 5$.
\item $J = (x^2,xy,xz,yz^2,z^3,ax+by+cz,dx+y^2)$, where $a,b,c,d \in S_1$ and $\h(x,y,z,b,c) = \h(x,y,z,c,d) = 5$.
\item $J = (x^2,xy,xz,y^3,z^3,ax+by+cz,dx+yz)$, where $a,b,c,d \in S_1$ and $\h(x,y,z,b,c) = \h(x,y,z,b,d) = \h(x,y,z,c,d) = 5$. 
\item All quadrics in $J$ can be written in terms of at most $6$ variables.  
\item All quadrics in $J$ generate an ideal of height at most $2$. 
\end{enumerate}
In particular, $\pd(S/J) \le 4$ for cases {\rm(1) - (5)}.
\end{prop}

\begin{proof} As $e(S/J) = 3$, $e(S/(J:\p)) = 1$ or $2$.

If $e(S/(J:\p)) = 1$, then $J:\p = (x,y,z)$ and $(x,y,z)^2 \subseteq J$.  If $J$ contains two linear forms, then an ideal of the form $(x,y,z^2)$ is contained in $J$, contradicting that $e(S/J) = 3$.  If $J$ contains a linear form, then $J$ must be $(x,y^2,yz,z^2)$.  If $J$ does not contain any linear forms, then $J$ must contain at least one more generator of the form $q=ax + by + cz$, which we may assume is degree $2$ (or else we are in case (6)).  If $\h(x,y,z,a,b,c) \ge 5$, then $(x,y,z)^2 + (ax+by+cz)$ is unmixed of multiplicity $3$ by Lemma~\ref{primary31} and then equals $J$.  This ideal fits into case (6) as well.  If $\h(x,y,z,a,b,c) = 4$ then we can modify $b,c$ modulo $(x,y,z)$ to assume $b$ and $c$ are multiple of $a$. Then $q=a\ell \in J$ for some linear form $\ell$ and then $J$ contains a linear form, $\ell$, a contradiction.

If $e(S/(J:\p)) = 2$, then by the previous proposition there are $4$ cases to consider:

\noindent \underline{Case I. $J:\p = (x,y,z^2)$}\\
In this case $(x^2, xy, xz, y^2, yz, z^3) \subseteq J \subseteq (x,y,z^2)$.  If $J$ contains $2$ independent linear forms, then $J = (x,y,z^3)$.  If $J$ contains exactly one linear form up to scalar multiplication, we may assume this form to be $x$, thus we have $(x, y^2, yz, z^3) \subseteq J$.  The ideal on the left has multiplicity $4$, so we need at least one additional generator of the form $ay + z^2$, which we may assume is a quadric, or else we are in case (7).  By Lemma~\ref{primary32}, $(x,y^2,yz,z^3,ay+z^2)$ is unmixed of multiplicity $3$ and so is equal to $J$.  If $\h(x,y,z,a) = 4$, then we are in case (2); if $a \in (x,y,z)$, then we are in case (1).  Finally we assume that $J$ does not contain a linear form.  We may assume that $J$ has two additional quadric generators of the form $ax + by + z^2, cx + dy$. Indeed, if two linearly independent quadrics are of the form $ax +by, cx + dy$, then $(ad-bc)y \in J$. Since $J$ contains no linear forms, we have $ad-bc\in \p$. Since $\p$ is a linear prime, then $S/\p$ is a UFD, then either one of $a, b, c, d \in \p$ or $a\equiv b$ and $c\equiv d$ modulo $\p$, or else $a\equiv c$ and $b\equiv d$ modulo $\p$. In the first case, we get a linear form in $J$.  In the latter case, after modifying $a,b,c,d$ modulo $\p$ we would get that $ax+by$ can be assumed to be a scalar multiple of $cx+dy$, which contradicts their linear independence. In the second case, after modifying $a,b,c,d$ modulo $\p$ we would obtain $a\ell \in J$ and $c\ell'\in J$ for linear forms $\ell, \ell'\in \p$. Since $a\notin \p$ (otherwise one of the generators in not minimal), then $\ell \in J$, giving a contradiction.

Finally, if $\h(x,y,z,c,d) = 5$, the ideal $(x^2, xy, xz, y^2, yz, ax+by+z^2, cx+dy)$ is equal to $J$ by Lemma~\ref{primary33}; otherwise, $c$ and $d$ have a common factor modulo $\p$ and so $J$ contains a linear form - a contradiction.

\noindent \underline{Case II. $J:\p = (x,y^2,yz,z^2,ay+bz)$}.\\
If $x \in J$, then $(x, y^3, y^2z, yz^2, z^3, y(ay+bz), z(ay+bz)) \subseteq J$ and there is at least one additional generator of the form $\alpha y^2 + \beta yz + \gamma z^2 + \delta(ay+bz)$, which we may assume is a quadric or else we are in case (7).  If $\delta = 0$ (which includes the case $\deg(a) = \deg(b) \ge 2$), then one checks that $(x,y,z)^2_\p \subseteq J_\p$ and hence $(x,y^2,yz,z^2) \subseteq J$ since $J$ is $\p$-primary.   S these ideals have the same multiplicity, $J= (x,y^2,yz,z^2)$.  This contradicts that $J:\p = (x,y^2,yz,z^2,ay+bz)$.  If $\delta \neq 0$ we may assume $\delta = 1$; then by a linear change of variables ($a' = a + \alpha y + \beta z$, $b' = b + \gamma z$), we may assume our additional quadric has the form $ay + bz$.  Then $J = (x, ay+bz, y^3, y^2z, yz^2,z^3)$ by Lemma~\ref{primary34}.  This ideal falls into case (7) as well.

If $x \notin J$, then $(x^2,xy,xz,y^3,y^2z,yz^2,z^3,y(ay+bz),z(ay+bz)) \subseteq J$.  There are at least $2$ additional quadric generators of the form $cx + \alpha y^2 + \beta yz + \gamma z^2 + \delta(ay + bz), dx + \epsilon y^2 + \zeta yz + \eta z^2$, where $c,d \in S_1$ and the remaining coefficients are in $K$, since otherwise we are in case (6).  If $\delta = 0$ for all such quadrics, then either there are at most $3$ such quadrics putting us in case (6), or else there is a quadric of the form $ex \in J$ with $e \notin (x,y,z)$.  It follows that $x \in J$, a contradiction.  Hence we may assume that $\delta = 1$  and that $\deg(a) = \deg(b) = 1$.  After a linear change of variables ($a' = a + \alpha y + \beta z$ and $b' = b + \gamma z$), we can assume that $\alpha = \beta = \gamma = 0$.  After a further linear change of variables we can assume that the second quadric has the form $dx + y^2$ or $dx + yz$.  (If $\epsilon = \zeta = \eta = 0$, then $x \in J$, contradicting that $J$ did not contain a linear form.)  Recall that we have $\h(x,y,z,a,b) = 5$. 

 Suppose the second quadric is $dx + y^2$.  If $\h(x,y,z,b,d) = 5$, then $J = (x^2,xy,xz,yz^2,z^3,cx+ay+bz,dx+y^2)$ by Lemma~\ref{primary35} putting us in case (4).  If $\h(x,y,z,b,d) = 4$, then after rewriting we may assume $d = 0$ or $d = b$.  If $d = 0$, then one checks that $J = (x,y,z)^2 + (cx+ay+bz)$, which is unmixed by Lemma~\ref{primary31}, putting us in case (6).  If $d = b$, then 
\[b(ax-yz) = a(bx+y^2) - y(cx+ay+bz)+cxy \in J,\]
and hence $ax - yz \in J$.  By Lemma~\ref{primary37} that $J = (x^2,xy,xz,z^3,bx+y^2,ax-yz,cx+ay+bz)$.  Again this falls into case (6). 

The case where the second quadric is $dx + yz$ is similar.  If $\h(x,y,z,a,d) = \h(x,y,z,b,d) = 5$, then $J = (x^2, xy, xz, y^3, z^3, cx+ay+bz, dx+yz)$ by Lemma~\ref{primary36}.  Otherwise we may assume one of the following: $d = a$, $d = b$ or $d = 0$.  If $d = a$ or $d = b$, then $J$ has the form of Lemma~\ref{primary37}.  If $d = 0$, then $x^2, y^2 \in J$ and $J$ has the form from Lemma~\ref{primary31}.  Both cases put us in case (6).

\noindent \underline{Case III. $J:\p = (x,y,z)^2+(ax+by+cz,dx+ey+fz)$}.\\
We have that $(x,y,z)^3 + (x,y,z)(ax+by+cz,dx+ey+fz) \subseteq J$, where $\h(x,y,z,I_2(\mathbf{M})) = 5$.  The ideal on the left has multiplicity $5$, so there are at least two additional quadric generators which, after a linear change of variables, we can take to be of the form $ax+by+cz + q, dx+ey+fz + q'$, where $q, q' \in (x,y,z)^2$.  (If there are not at least $2$ such quadrics, then all quadrics in $J$ can be written in terms of at most $6$ variables.)  After relabeling $a,\ldots,f$ we can assume $q, q' = 0$ without changing $\h(I_2(\mathbf{M}))$ or $\h(I_1(\mathbf{M}))$.  Hence $(x,y,z)^3 + (ax+by+cz, dx+ey+fz) \subseteq J$.  If $\h(x,y,z,I_2(\mathbf{M})) \ge 5$, where 
\[\mathbf{M} = \begin{pmatrix} a&b&c\\d&e&f \end{pmatrix},\]
then $\h(x,y,z,a,b,c,d,e,f) \ge 5$.  If in addition $\h(x,y,z,a,b,c,d,e,f) \ge 6$, then this ideal equals $J$ by Lemma~\ref{primary38}.  This ideal falls into case (7).  If $\h(x,y,z,a,b,c,d,e,f) = 5$, then by Lemma~\ref{matrixlemma2} we may take $\mathbf{M}$ to be of the form
\[\mathbf{M} = \begin{pmatrix} a&b&0\\0&a&b \end{pmatrix},\]
modulo $(x,y,z)$.  In other words, $ax + by + q, ay + bz + q' \in J$, where $q, q' \in (x,y,z)^2$.
In this case
\[a(xz - y^2) = z(ax+by + q) - y(ay+bz + q') - zq + yq'.\]
Since $(x,y,z)^3 \subseteq J$, we have $a(xz-y^2) \in J$.  Since $a \notin (x,y,z)$ and $J$ is $(x,y,z)$-primary, $xz - y^2 \in J$.  Now $J = (x,y,z)^3 + (ax+by, ay+bz, xz + y^2)$ by Lemma~\ref{primary310}.  This ideal falls into case (6).

\noindent \underline{Case IV. Quadrics in $J:\p$ can be written in terms of $6$ variables}.\\
Since $J \subseteq J:\p$, all quadrics in $J$ can be written in terms of $6$ variables, putting us in case (6).
\end{proof}

\begin{prop}\label{1;4} Let $J$ be $\p$-primary with $e(S/J) = 4$.  Then at least one of the following holds:
\begin{enumerate}
\item All quadrics in $J$ can be written in terms of at most $6$ variables.
\item All quadrics in $J$ generate an ideal of height at most $2$.
\item $J$ contains a linear form.
\item $\pd(S/J) \le 5$.
\end{enumerate}
\end{prop}

\begin{proof} We can assume $J$ does not contain a linear form, since otherwise we are in case (3).
The ideal $J:\p$ is $\p$-primary and $e(S/(J:\p))$ must be $1$, $2$ or $3$. \\
\noindent \underline{Case I. $e(S/(J:\p)) = 1$}.\\
In this case $J:\p = (x,y,z)$ and hence $J = (x,y,z)^2$.  This fall into case (1).  \\
\noindent \underline{Case II. $e(S/(J:\p)) = 2$}.\\
There are $3$ subcases to consider corresponding to Lemma~\ref{1;2}. \\
\noindent \underline{Case II.A. $J:\p = (x,y,z^2)$}.\\
We have $(x^2,xy,xz,y^2,yz,z^3) \subseteq J\subseteq J:\p$.  There is another generator which we may assume is a quadric of the form $ax + by + \alpha z^2$, where $a, b \in S_1$ and $\alpha \in K$.  We may further assume that $\alpha = 1$, otherwise all quadrics in $J$ are contained in $(x,y)$, putting us in case (2).  This ideal is then unmixed of multiplicity $4$ by Lemma~\ref{primary42} and so equal to $J$.  Note that all quadrics in $J$ can be written in terms of the linear forms $x,y,z,a,b$.\\
\noindent \underline{Case II.B. $J:\p = (x,y^2,yz,z^2,ay+bz)$, where $\h(x,y,z,a,b) = 5$}.\\
Here we have $(x^2,xy,xz,y^3,y^2z,yz^2,z^3,y(ay+bz),z(ay+bz)) \subseteq J\subseteq J:\p$.  We may assume $J$ contains a quadric of the form $cx + \alpha y^2 + \beta yz + \gamma z^2 + (ay+bz)$.   After a linear change of variables we may assume (without changing $\h(x,y,z,a,b)$) that the quadric has the form $cx + ay + bz$.  By assumption $\h(x,y,z,a,b) = 5$, so this ideal is unmixed and equal to $J$ by Lemma~\ref{primary43}.  All quadrics in $J$ can be written in terms of  the linear forms $x,y,z,a,b,c$ putting us in case (1).\\
\noindent \underline{Case II.C. $J:\p = (x,y,z)^2 + (ax+by+cz,dx+ey+fz)$, where}\\\underline{ $\h(x,y,z,I_2(\mathbf{M})) = 5$}.\\
Since $(x,y,z)^3 + (x,y,z)(ax+by+cz,dx+ey+fz) \subseteq J\subseteq J:\p$, it is easy to see that either all quadrics in $J$ can be written in terms of at most $6$ variables, or else, after relabeling the coefficients (and without affecting $\h(x,y,z,I_2(\mathbf{M}))$), we may assume $ax+by+cz,dx+ey+fz \subseteq J$. These two forms are linearly independent elements of $\p_{\p}\setminus \p_{\p}^2$.    Since $(x,y,z)^3 \subseteq J$, we have  $HF_{S_{\p}/J_{\p}} = (1,1,1,0,\ldots)$; in particular $e(S/J) \leq 3$, contradicting that $e(S/J) = 4$.\\
\noindent \underline{Case III. $e(S/(J:\p)) = 3$}.\\
In this case $J:\p$ is one of cases (1) - (7) from the previous proposition.  We are done if $J:\p$ is as in cases (6) or (7) or if $J$ contains a linear form by Lemma~\ref{LinForm}.  \\
\noindent \underline{Case III.A. $J:\p = (x,y^2,yz,z^2)$}.\\
In this case $(x^2,xy,xz,y^3,y^2z,yz^2,z^3) \subseteq J\subseteq J:\p$ and $J$ contains additional generators of the form $ax + \alpha y^2 + \beta yz + \gamma z^2$ where $a \in S_1 - \p$.  Either all quadrics in $J$ can be written in terms of at most $6$ variables, or there are at least $4$ such generators with linearly independent $x$-coefficients.  After taking a linear combination of these $4$ generators, we have $a'x \in J$ for some $a'\notin \p$ and hence $x \in J$, contradicting that $J$ does not contain a linear form.\\
\noindent \underline{Case III.B. $J:\p = (x,y^2,yz,z^3,ay+z^2)$, where $\h(x,y,z,a) = 4$}.\\
We have $(x^2,xy,xz,y^3,y^2z,yz^2,ayz+z^3) \subseteq J\subseteq J:\p$. We may assume that $J$ has additional quadric generators of the form $b_ix + \alpha_i yz + \beta_i y^2 +\gamma_i (ay+z^2)$.  We have three possibilities to consider:
\begin{enumerate}
\item $\h(x,y,z,a,b_1,b_2,\ldots) \le 6$, in which case all quadrics in $J$ can be expressed in terms of at most $6$ variables.
\item One of extra quadric generators has the form $bx$, with $b \notin (x,y,z)$.  Since $J$ is $(x,y,z)$-primary, it follows that $J$ contains the linear form $x$, a contradiction.
\item If neither of the above hold, we must have exactly $3$ additional quadrics.  If there are fewer than $3$ such quadrics, we are in case (1).  If there are more than $3$ then after taking linear combinations of the generators we can produce a generator of the form $bx$ and are in case (2).  Moreover, after taking linear combinations of these $3$ quadrics, we may assume they have the following form: $b_1x + ay+z^2, b_2x+yz, b_3x + y^2$ with $\h(x,y,z,a,b_1,b_2,b_3) = 7$.  In this case, one notes that
\[ab_3x = a(b_3x+y^2) - y(b_1x+ay+z^2) + b_1xy + yz^2 \in J.\]
As $ab_3 \notin (x,y,z)$, we again have the linear form $x \in J$.
\end{enumerate}
\noindent \underline{Case III.C. $J:\p = (x^2,xy,xz,y^2,yz,ax+by+z^2,cx+dy)$, where}\\\underline{$\h(x,y,z,c,d) = 5$}.\\
We have $(x,y,z)(x^2,xy,xz,y^2,yz,ax+by+z^2,cx+dy) \subseteq J\subseteq J:\p$.  If it is not the case that all quadrics in $J$ are expressible in terms of at most $6$ variables, then after a relabeling of linear forms, we may assume that $ax+by+z^2, cx+dy \in J$.  One checks then that $(ad-bc)(x^2,xy,y^2) \subseteq J$.  If $ad-bc \in \p$, then ${\rm ht}(x,y,z,a,b,c,d) \le 6$ putting us in case (6).  If $ad-bc\notin \p$, then since $J$ is $\p$-primary, we have $J_1=(x^2,xy,y^2,ax+by+z^2,cx+dy) \subseteq J$.  This ideal is unmixed of multiplicity $4$ by Lemma~\ref{primary44}, $J=J_1$ and we have $\pd(S/J) = 4$.\\
\noindent \underline{Case III.D. $J:\p = (x^2,xy,xz,yz^2,z^3,cx+ay+bz,dx+y^2)$, where}\\\underline{$\h(x,y,z,b,c) =  \h(x,y,z,c,d) = 5$}.\\
Once again either we are in case (6) or we may assume $cx+ay+bz,dx+y^2 \subseteq J$ and ${\rm ht}(x,y,z,a,b,c,d) = 7$.  Since $dx^2 = x(dx+y^2) - y(xy) \in J$ and $d \not\in \p$, we also have $x^2 \in J$. Then 
$J_1=(x^2, xz^2, xy^2, y^4, z^4, dx+yz, cx+ay+bz)\subseteq J\subseteq J:\p$. One checks that $J_1$ is unmixed of multiplicity $4$ and hence equal to $J$ and that $\pd(S/J) = 4$. \\
\noindent \underline{Case III.E. $J:\p = (x^2,xy,xz,y^3,z^3,cx+ay+bz,dx+yz)$},where\\\underline{$\h(x,y,z,b,c) = \h(x,y,z,b,d) =  \h(x,y,z,c,d) = 5$}.\\
This case is nearly identical to the previous case, so we leave the proof to the reader.

This completes the proof.
\end{proof}

\section{$\langle2;2\rangle$ structures}\label{22struct}

 The section contains our characterization of certain $\q$-primary ideals, where $\q=(x,y,q)$ is a height three prime generated by $2$ linear forms $x,y$ and a quadric form $q$.  For our purposes, we need to characterize the possible $\langle 2;2\rangle$ structures; that is, $\q$-primary ideals $J$ with $e(S/J) = 4$.  
  
\begin{prop}\label{st2;2}
Let $\q=(x,y,q)$ be a height $3$ multiplicity $2$ prime and let $J$ be a $\q$-primary ideal with $e(S/J) = 4$. Further assume that $J$ contains a height $3$ ideal generated by quadrics.  Then $J$ has one of the following forms:
\begin{enumerate}
\item $J=(x,y^2,q),$ 
\item $J=(x^2,xy,y^2,q) + L',$ where $L' $ is generated in degrees $\ge 3$,
\item $J=(x^2,xy,y^2,ax+by,q) + L'$, where $a,b \in S_1$, ${\rm ht}(x,y,a,b,q) = 5$ and $L' $ is generated in degrees $\ge 3$,
\item $J=(x^2, xy, y^2, ax+by, cx+dy, ad-bc+ex+fy = q),$ where $a,b,c,d,e,f \in S_1$, ${\rm ht}(x,y,a,b) = {\rm ht}(x,y,c,d) = 4$ and ${\rm ht}(x,y,ad - bc) = 3$.
\end{enumerate}
\end{prop}

\begin{proof}
First, notice that after possibly modifying $q$ modulo $(x,y)$ we may assume $q\in J$. Indeed, if not, every quadric in $J$ would be contained in $(x,y)$, a contradiction. Notice that the Hilbert function of $J$ locally at $\mathfrak{q}$ is $HF_{S_\mathfrak{q}/J_\mathfrak{q}}: (1, 1,0,\ldots)$. In particular, $\mathfrak{q}_\mathfrak{q}^2\subseteq J_\mathfrak{q}$. Since $J$ is $\mathfrak{q}$-primary, this yields $\mathfrak{q}^2\subseteq J \subseteq \mathfrak{q}$.  We conclude that $(x^2,xy,y^2,q)\subseteq J$.

Assume $J$ also contains a linear form. Without loss of generality we may assume it is $x$. Then  $(x,y^2,q)\subseteq J$. However, $x,y^2,q$ is a complete intersection ideal contained in the unmixed ideal $J$ and $e(S/(x,y^2,q)) = e(S/J) = 4$; hence $J=(x,y^2,q)$ by Lemma~\ref{unmixedequal}, and we are in case $(1)$.

We now assume $J$ does not contain a linear form. Since $e(S/(x^2,xy,y^2,q))=6$, there must be at least an additional generator for $J$. If all such additional generators have degree at least three, then we are in case $(2)$ of the statement. We may then assume $J$ contains a quadric generator $q'\notin (x^2,xy,y^2,q)$. Since $q\in J$, without loss of generality we may assume $q'$ is of the form $q' = ax+by$, where $a,b$ are linear forms. Since $(x,y)^2\subseteq J$ we may assume ${\rm ht}(x,y,a,b) \ge 3$. Also, if ${\rm ht}(x,y,a,b) = 3={\rm ht}(x,y,a)$, then it is easily seen that $a\ell \in J$ for some $\ell \in (x,y)$, and hence $\ell \in J$, contradicting that $J$ contains no linear forms. Thus we may assume  ${\rm ht}(x,y,a,b)=4$.

If ${\rm ht}(x,y,a,b,q) = 5$ and there are no additional quadric generators of $J$, then we are in case $(3)$.  If ${\rm ht}(x,y,a,b,q) = 4$, then we may write $q = fx + ey + da - cb$ for  linear forms $c,d,e,f$.  Then $b(cx + dy) = -x(fx + ey + da - cb) + f(x^2) + e(xy) + d(ax + by) \in J$.  Since $b \notin \q$ and $J$ is $\q$-primary, $cx + dy \in J$.  If ${\rm ht}(x,y,c,d) = 2$, then $q \in (x,y)$, a contradiction.  If ${\rm ht}(x,y,c,d) = 3$, then $c,d$ share a common factor modulo $(x,y)$ and again we have a linear form in $J$.

Hence, we may assume there is an additional quadric generator $q''=cx+dy \in J - (x^2,xy,y^2,q,q')$, where $c,d$ are linear forms with ${\rm ht}(x,y,c,d) = 4$.  Note that $(ad - bc)x = d(ax+by) - b(cx + dy) \in J$.  If $ad - bc \not\in \q$, then $x \in J$, contradicting that $J$ does not contain any linear forms.  Hence we can write $ad - bc = ex + fy + \alpha q$ for linear forms $e,f$ and $\alpha \in k$.

If $\alpha = 0$, then $ad - bc \in (x,y)$.  So $ad - bc = 0 (\textrm{modulo } (x,y))$.  It follows that  $ax+by = \alpha(cx + dy)$ modulo $(x,y)^2$, which contradicts that $q'' = cx + dy \notin (x^2, xy, y^2, q, q' = ax+by)$.  Hence $\alpha \not= 0$, and we may assume that $q = ad - bc + ex + fy$, for linear forms $e,f$.

Set $J' = (x^2, xy, y^2, ax+by, cx+dy, ad-bc+ex+fy)$.  By Lemma~\ref{22gen}, $J'$ has multiplicity $4$ and is Cohen-Macaulay; in particular, $J'$ is unmixed, ${\rm ht}(J') = {\rm ht}(J) = 3$ and $J' \subseteq J$;  hence $J' = J$ by Lemma~\ref{unmixedequal}.
\end{proof}

\section{Multiplicity $2$}\label{e2}

In this section, we again set $I = (q_1,q_2,q_3,q_4)$ to be an ideal generated by $4$ quadrics with $\h(I) = 3$.  The next proposition gives a bound on the projective dimension when $e(S/I) = 2$.

\begin{prop}\label{mult2} Suppose $e(S/I) = 2$.  Then $\pd(S/I) \le 6$.
\end{prop}

\begin{proof} There are three cases for the type of $I^{un}$: $\langle 2;1 \rangle$, $\langle 1,1;1,1 \rangle$, and $\langle 1;2 \rangle$.  

If $I^{un}$ is of type $\langle 2;1 \rangle$, then $I^{un}$ is prime and hence is Cohen-Macaulay by Corollary~\ref{e=2}.  By Lemma~\ref{iuncm}, $\pd(S/I) \le 4$.

If $I^{un}$ is of type $\langle 1,1;1,1 \rangle$, then $I^{un} = (u,v,w) \cap (x,y,z)$, hence  $\pd(S/I) \le 6$ by Corollary~\ref{iungens}. 

If $I^{un}$ is of type $\langle 1;2 \rangle$, then $I^{un}$ satisfies one of the four cases in Proposition~\ref{1;2}.  In the first two cases, $I^{un}$ contains a linear form and $\pd(S/I) \le 5$ by Lemma~\ref{LinForm}.  In the last case, we have $\pd(S/I) \le 6$ by Lemma~\ref{regseq} again.  Finally, if $I^{un} = (x,y,z)^2 + (ax+by+cz,dx+ey+fz)$, then  by Lemma~\ref{primary21}, $\pd(S/(x^2,y^2,z^2):I^{un}) = 4$.  Hence $\pd(S/I) \le 5$ by Lemma~\ref{ses}.
\end{proof}

\section{Multiplicity $3$ and $5$}\label{e35}

Recall that we are assuming $I = (q_1, q_2, q_3, q_4)$ is a height $3$ ideal generated by four quadrics and $(q_1, q_2, q_3)$ form a complete intersection.  As the cases $e(S/I) = 3$ and $e(S/I) = 5$ are very similar, we will handle these cases simultaneously.  When $e(S/I) = 3$, we show that $\pd(S/I) \le 6$.  When $e(S/I) = 5$, we consider the ideal $L = (q_1,q_2,q_3):I$, which is directly linked to $I^{un}$ and satisfies $e(S/L) = 3$.  In this case, we show that either $\pd(S/I) \le 6$ or the quadrics in $L$ are extended from a polynomial ring with at most $6$ variables.  Thus either $I^{un}$ or $L$ is of one of the following types: $\langle 3;1 \rangle$, $\langle 1;3 \rangle$, $\langle 1,2;1,1 \rangle$, $\langle 1,1;1,2 \rangle$, $\langle 1,1,1;1,1,1 \rangle$.  For several of these, we devote a separate lemma.  The final bound arguments are collected in Proposition~\ref{mult35}.

\begin{lem}\label{pd13}
If $I^{un}$ is type $\langle 1;3 \rangle$, then $\pd(S/I) \le 6$.  If $L$ is of type $\langle 1;3 \rangle$, then $\pd(S/I) \le 6$ or the quadrics in $L$ can be expressed in terms of at most $6$ variables.
\end{lem}

\begin{proof} Note that we are done if either $I^{un}$ or $L$ contains a linear form (by Lemma~\ref{LinForm}), or if the quadrics in either are expressible in terms of at most $6$ variables (by Lemma~\ref{regseq}), or generate an ideal of height at most $2$ (which contradicts our assumptions about $I$).  By Proposition~\ref{1;3}, this leaves only cases (3) - (5) to consider.  If $L$ is one of these ideals, then $\pd(S/L) \le 4$ by Lemmas~\ref{primary33}, \ref{primary35} and \ref{primary36}, and so $\pd(S/I) \le 5$ by Lemma~\ref{ses}.  If $I^{un}$ is one of these ideals, then we note that again by Lemmas~\ref{primary33}, \ref{primary35} and \ref{primary36}, an ideal $L'$ directly linked to $I^{un}$ satisfies $\pd(S/L') \le 4$.  Hence $\pd(S/I) \le 5$ by Lemmas~\ref{doublelink} and \ref{ses}.
\end{proof}

\begin{lem}\label{pd1211}
If $I^{un}$ is of type $\langle 1,2;1,1 \rangle$, then $\pd(S/I) \le 6$.  If $L$ is of type $\langle 1,2;1,1 \rangle$, then $\pd(S/I) \le 6$ or  the quadrics in $L$ can be expressed in terms of at most $6$ variables.
\end{lem}

\begin{proof} First consider the case $L = (x,y,z) \cap (u,v,q)$, where $x,y,z,u,v \in S_1$ and $q \in S_2$.  By Lemma~\ref{inters}, we may assume that $q \in (x,y,z)$.  Thus $(x,y,z)+(u,v,q) = (x,y,z,u,v)$.  From Lemma~\ref{sum}
it follows that $\pd(S/L) \le 4$.  Hence $\pd(S/I) \le 5$ by Lemma~\ref{ses}.

Now if $I^{un} = (x,y,z) \cap (u,v,q)$, we may again assume that $q \in (x,y,z)$, say $q = ax+by+cz$.  If $\h(x,y,z,u,v) \le 4$, then $I^{un}$ contains a linear form and we are done.  We may then assume that $\h(x,y,z,u,v) = 5$ and so $I^{un} = (q,xu,xv,yu,yv,zu,zv)$.  Let $\partial_i$ denote the $i^{\text th}$ differential map in the minimal free resolution of $I^{un}$.  By Lemma~\ref{1211case},  we have $\pd(\coker(\partial_4*)) \le 5$.  Hence by Lemma~\ref{messy}, $\pd(S/I) \le 5$.
\end{proof}

\begin{lem}\label{pd1112}
If $I^{un}$ is of type $\langle 1,1;1,2 \rangle$, then $\pd(S/I) \le 6$.  If $L$ is of type $\langle 1,1;1,2 \rangle$, then $\pd(S/I) \le 6$ or  the quadrics in $L$ can be expressed in terms of at most $6$ variables.
\end{lem}

\begin{proof} Suppose that $L$ or $I^{un}$ is $(u,v,w) \cap L_2$, where $L_2$ is as in Proposition~\ref{1;2}.  First assume $L_2 = (x,y,z^2)$. If $L=(u,v,w) \cap (x,y,z^2)$, then since $\pd(S/(u,v,w,x,y,z^2)) \le 6$ then we have $\pd(S/L) \le 5$ by Lemma~\ref{sum}.  Thus $\pd(S/I) \le 6$ by Lemma~\ref{ses}.  If $I^{un} = (u,v,w) \cap (x,y,z^2)$, then $\pd(S/I) \le 6$ by Corollary~\ref{iungens}.

Assume then that $L_2 = (x,y^2,yz,z^2,ay+bz)$. By Lemma \ref{LinForm} we may assume that $\h(u,v,w,x) = 4$ or else $(u,v,w)\cap L_2$ contains a linear form. 
We may further assume that $\deg(a) = \deg(b) = 1$ and that $ay+bz \in (u,v,w)\cap L_2$, or else all quadrics in $(u,v,w)\cap L_2$ are expressible in terms of $u,v,w,x,y,z$.  If $\h(u,v,w,x,y,z) = 6$, then $(u,v,w) \cap L_2 = (ay+bz) + (u,v,w)(x,y^2,yz,z^2)$ all of whose quadrics are contained in the height $2$ ideal $(ay+bz,x)$, contradicting that $\h(q_1, q_2, q_3) = 3$.  So we must have $\h(u,v,w,x,y,z) \le 5$ and may assume that $y = u$.  Since $au+bz \in (u,v,w)\cap L_2 \subseteq (u,v,w)$, we must have $b \in (u,v,w)$ or $z \in (u,v,w)$.  In either case $\h(u,v,w,x,y,z,a,b) \le 6$.  Hence all quadrics in $(u,v,w) \cap L_2$ are expressible in terms of at most $6$ variables.

Suppose $L_2 = (x,y,z)^2 + (ax+by+cz,dx+ey+fz)$.  After rewriting, we may assume that 
$ax+by+cz, dx+ey+fz \in (u,v,w)$, since otherwise all quadrics in $(u,v,w) \cap L_2$ could be expressed in terms of at most $6$ variables.  Thus $(u,v,w) + L_2 = (u,v,w) + (x,y,z)^2$, which has projective dimension at most $6$.  If $L = (u,v,w) \cap L_2$, then $\pd(S/L) \le 5$ by Lemma~\ref{sum}, and hence $\pd(S/I) \le 6$ by Lemma~\ref{ses}.

The last case to consider is when $I^{un} = (u,v,w) \cap L_2$.  We may still assume $ax+by+cz, dx+ey+fz \in (u,v,w)$.  
If $\h(u,v,w,x,y,z) = 6$, it follows that $a,b,c,d,e,f \in (x,y,z,u,v,w)$.  Thus all quadrics in $I^{un}$ are expressible in terms of $6$ variables.  If $\h(u,v,w,x,y,z) = 5$, we may assume that $u = x$.  

In the above situation, we can take $I = (x \ell_1, x \ell_2, ax+by+cz, dx+ey+fz)$ and $e(S/I) = 3$ for linear forms $\ell_1, \ell_2$.  But clearly $I + (x) = (x, by + cz, ey + fz)$.  Since ${\rm ht}(I_2) = 2$, where $I_2$ is as in Proposition~\ref{1;2}, we must have ${\rm ht}(I + (x)) = 3$.  Therefore $e(S/(I + (x))) = 4$.  This cannot happen since $I + (x) \supseteq I$ and $e(S/I) = 3$.  \\

If $\h(u,v,w,x,y,z) = 4$, we may assume that $u = x$ and $v = y$.  As above, we may assume $ax+by+cz,\,dx+ey+fz \in (x,y,w)$, hence we may assume $c = w$ and $f=0$, since $w,x,y,z$ is a regular sequence (the case where both $c,f$ are non-zero multiples of $w$ can be reduced to the above, by taking a linear combination of the two generators, and the case where $c=f=0$ cannot happen otherwise the height of the ideal of minors $I_2$ is $1$ modulo $(x,y,z)$, where $I_2$ is as in Proposition \ref{1;2}). It follows that $I^{un} = (x^2,xy,xz,y^2,yz,z^2w,ax+by+wz,dx+ey)$.  By the height restriction on $I_2$, we have ${\rm ht}(x,y,z,d,e) = 5$. By Lemma~\ref{1112case}, $\pd(\coker(\partial_4*)) \le 5$, where $\partial_4$ is the $4$th map in the free resolution of $I^{un}$.  By Proposition~\ref{messy}, $\pd(S/I) \le 5$. 
\end{proof}

\begin{lem}\label{pd111111}
If $I^{un}$ is type $\langle 1,1,1;1,1,1 \rangle$, then $\pd(S/I) \le 6$.  If $L$ is of type $\langle 1,1,1;1,1,1 \rangle$, then $\pd(S/I) \le 6$ or  the quadrics in $L$ can be expressed in terms of at most $6$ variables.
\end{lem}

\begin{proof} Write $J = \bigcap_{i = 1}^3 L_i$ where $L_i = (x_i,y_i,z_i)$.  If $\h(L_i + L_j) = 6$ for any $i \neq j$, then all quadrics in $J$ are expressible in terms of $x_i,y_i,z_i,x_j,y_j,z_j$.  Otherwise, $\h(L_i+L_j) \le 5$ for all pairs $i \neq j$.  Unless there is a linear form in $J$, this forces $\h(L_1 + L_2 + L_3) \le 6$.  Thus either $I^{un}$ or $L$ contain a linear form, yielding $\pd(S/I) \le 5$ by Lemma~\ref{LinForm}, or all generators are expressible in terms of at most $6$ variables.  This means $\pd(S/L) \le 5$, $L$ being an unmixed ideal extended from a  $6$ variable polynomial ring, and hence $\pd(S/I) \le 6$ by Lemma~\ref{ses}.  The case for $I^{un}$ follows from Collolary~\ref{iungens}.
\end{proof}

\begin{prop}\label{mult35}
If $e(S/I) = 3$, then $\pd(S/I) \le 6$.  If $e(S/I) = 5$, then $\pd(S/I) \le 6$ or all quadrics in $L = (q_1,q_2,q_3):I$ can be expressed in terms of at most $6$ variables.
\end{prop}

\begin{proof} If $e(S/I) = 3$, then $I^{un}$ or $L = (q_1,q_2,q_3):I$ is of one of the five types listed at the beginning of the section.  If either $I^{un}$ or $L$ is of type $\langle 3;1 \rangle$, then it contains a linear form by Theorem~\ref{e=3}.  By Lemma~\ref{LinForm}, $\pd(S/I) \le 5$.  The other four cases are covered by Lemmas~\ref{pd13}, \ref{pd1211}, \ref{pd1112} and \ref{pd111111}.
\end{proof}

\section{Multiplicity $4$}\label{e4}

In this section we cover the case when $e(S/I) = 4$.  Recall that $I = (q_1,q_2,q_3,q_4)$, $\h(I) = \h(q_1,q_2,q_3) = 3$.  We again set $L = (q_1,q_2,q_3):I$ and note that $L$ is unmixed and $e(S/L) = 4$.  Hence our strategy for bounding the projective dimension of $I$ is one of the following:
\begin{enumerate}
\item Show $\pd(S/L) \le 5$, yielding $\pd(S/I) \le 6$ by Lemma~\ref{ses}.
\item Show that $L$ contains a linear form and hence $\pd(S/I) \le 5$ by Lemma~\ref{LinForm}.
\item Show that all quadrics in $L$ can be written in terms of at most $6$ variables.
\end{enumerate}
The last case is completed in the following section where we show that if $e(S/I) = 4$ or $5$ and any complete intersection of quadrics in $I$ can be written in terms of at most $6$ variables, then $\pd(S/I) \le 6$.

We handle several of the cases for the type of $L$ separately in the following lemmas.

\begin{lem}\label{pd1311}
If $e(S/L) = 4$ and $L$ is of type $\langle 1,3;1,1 \rangle$, then $\pd(S/I) \le 5$.
\end{lem}

\begin{proof} We have $L = (u,v,w) \cap L_2$, where $L_2$ is generated by a linear form $x$ and the $2\times 2$ minors of a $2\times3$ matrix $\mathbf{M}$ of linear forms.  We may assume that $\h(x,u,v,w) = 4$, since otherwise $L$ contains a linear form.  We may further assume that there are two quadrics $\Delta_1, \Delta_2 \in I_2(\mathbf{M}) \cap (u,v,w)$ that are linearly independent in modulo $(x)$, since otherwise the quadrics in $L$ generate an ideal of height at most $2$.  It follows that $J = (u,v,w) + L_2$ is generated by $4$ linear forms and a quadric form.  Hence $\pd(S/J) \le 5$.  By Lemma~\ref{sum}, $\pd(S/L) \le 4$ and by Lemma~\ref{ses}, $\pd(S/I) \le 5$.
\end{proof}

\begin{lem}\label{pd1113}
If $e(S/L) = 4$ and $L$ is of type $\langle 1,1;1,3 \rangle$, then $\pd(S/I) \le 6$ or the quadrics in $L$ can be written in terms of at most $6$ variables.
\end{lem}

\begin{proof}
Here $L = (u,v,w) \cap L_2$, where $u,v,w \in S_1$ and $L_2$ is as in Proposition~\ref{1;3}.  If $L_2 = (x,y^2,yz,z^2)$, then $(u,v,w) + L_2 = (u,v,w,x,y^2,yz,z^2)$.  Hence $\pd(S/(L_2 + (u,v,w))) \le 6$ and by Lemmas~\ref{sum} and \ref{ses} we have $\pd(S/I) \le 6$. 

 If $L_2 = (x,y^2,yz,z^3,ay+z^2)$, then again we consider $L_2 + (u,v,w)$.  If $\h(u,v,w,x,y,z,a) = 7$, then the quadrics in $L$ are contained in $(x)$, which is not possible.  If $\h(u,v,w,x,y,z,a) \le 6$, then $\pd(S/((u,v,w)+L_2)) \le 6$ and we are done similarly to the above.

If $L_2 = (x^2,xy,xz,y^2,yz,ax+by+z^2,cx+dy)$, then after a linear change of coordinates we may assume that $ax+by+z^2,cx+dy \in (u,v,w)$, since otherwise all quadrics in $L$ are expressible in terms of at most $6$ variables.  It follows that $\pd(S/((u,v,w)+L_2)) \le 6$ and hence $\pd(S/I) \le 6$.  A similar argument works if $L_2$ is as in cases (4) and (5) of Proposition~\ref{1;3}.  In cases (6) or (7) there is nothing to show.
\end{proof}

\begin{lem}\label{pd4cases}
Suppose $e(S/L) = 4$ and $L$ is of type $\langle 2,2;1,1 \rangle$, $\langle 1,2;2,1 \rangle$, $\langle 1,1;2,2 \rangle$, $\langle 1,1,2;1,1,1 \rangle$, $\langle 1,1,1;1,1,2 \rangle$, or $\langle 1,1,1,1;1,1,1,1 \rangle$. Then either $\pd(S/I) \le 6$ or any three quadrics in $L$ can be written in terms of at most $6$ variables.
\end{lem}

\begin{proof} 
By the associativity formula \ref{Ass}, an unmixed height three ideal of multiplicity 2 is either prime, primary to a linear prime, or the intersection of two linear primes. Then, by the results of the previous section, it has one of the following $5$ forms:
\begin{itemize}
\item $(x,y,q)$, where $x,y \in S_1$ and $q \in S_2$ (this includes the case $(x,y,z) \cap (x,y,w)$)
\item $(x,a,b) \cap (x,c,d)$, where $x,a,b,c,d \in S_1$ and $\h(x,a,b,c,d) = 5$
\item $(u,v,w) \cap (x,y,z)$, where $u,v,w,x,y,z \in S_1$ and $\h(u,v,w,x,y,z) = 6$
\item $(x,y^2,yz,z^2,ay+bz)$, where $x,y,z,a,b \in S_1$, $\h(x,y,z,a,b) = 5$ 
\item $(x,y,z)^2 + (ax+by+cz,dx+ey+fz)$, where $x,y,z,a,b,c,d,e,f \in S_1$ and $\h(x,y,z,I_2) = 5$, where $I_2$ denotes the ideal of $2\times 2$ minors as in Proposition~\ref{1;2}.
\end{itemize}
Therefore, an ideal $L$ as in the statement can be written as $L=L_1 \cap L_2$, where $L_1$ and $L_2$ are ideals of one of the above $5$ forms.

Note that if either $L_1$ or $L_2$ fall into the third case, then all quadrics in $L$ can be expressed in terms of $6$ variables.  So we may ignore this case.  In the fourth case, we add the assumption that $a,b \in S_1$ since if $\deg(a), \deg(b) \ge 2$, then $q_1,q_2,q_3 \in (x,y^2,yz,z^2)$ and require at most $6$ variables to express.  We handle each of the other $10$ possible pairings separately.\\
\noindent \underline{Case I. $L_1 = (x,y,q)$, $L_2 = (w,z,q')$}.\\
Here $L_1 + L_2$ is generated by at most $4$ linear forms and $2$ quadric forms.  Hence $\pd(S/(L_1+L_2)) \le 6$.  By Lemma~\ref{sum}, $\pd(S/L) \le 5$.  By Lemma~\ref{ses}, $\pd(S/I) \le 6$.\\
\noindent \underline{Case II. $L_1 = (x,y,q)$, $L_2 = (z,ac,ad,bc,bd)$}.\\
By Lemma~\ref{inters}, we may assume that $q \in L_2$.  If $\h(x,y,z,a,b,c,d) = 7$, then all quadrics in $L$ are contained in $(z,q)$, contradicting that $\h(q_1,q_2,q_3) = 3$.  If $\h(x,y,z,a,b,c,d) \le 6$, then $\pd(S/(L_1+L_2)) \le 6$ and hence $\pd(S/I) \le 6$ by Lemmas~\ref{sum} and \ref{ses}.  \\
\noindent \underline{Case III. $L_1 = (x,y,q)$, $L_2 = (z,v^2,vw,w^2,av+bw)$}.\\
This is very similar to the previous case.  We may again assume $q \in L_2$.  If $\h(x,y,z,v,w,a,b) = 7$, then all quadrics in $L$ are contained in $(z,q)$, a contradiction.  Otherwise $\h(x,y,z,v,w,a,b) \le 6$, in which case  $\pd(S/(L_1+L_2)) \le 6$, whence $\pd(S/I) \le 6$ as before.\\
\noindent \underline{Case IV. $L_1 = (u,v,q)$, $L_2 = (x,y,z)^2 + (ax+by+cz,dx+ey+fz)$}.\\
After a possible linear change of variables, we may assume $ax+by+cz, dx+ey+fz \in L_1$, or else all quadrics in $L$ can be written in terms of at most $6$ variables.  Then $L_1 + L_2 = (x,y,z)^2 + (u,v,q)$ with $q \in (x,y,z)$.  Write $q = a'x+b'y+c'z$.  If ${\rm ht}(x,y,z,u,v,a',b',c') \le 6$, then $\pd(S/(L_1 + L_2)) \le 6$ and we are done.  If ${\rm ht}(x,y,z,u,v,a',b',c') \ge 7$, then one checks that $\pd(S/(L_1 + L_2)) = 6$.  

\noindent \underline{Case V. $L_1 = (x,ac,ad,bc,bd)$, $L_2 = (y,eg,eh,fg,fh)$}.\\
Here $L$ is the intersection of four linear primes.  If pairwise, the height of $6$ linear generators of $2$ primes is $6$, then all quadrics in $L$ could be written in terms of those $6$ linear forms only.  So we may assume that every pair of linear primes has a linear form in its intersection.  If no three linear primes have a common linear form, then we must have $4$ more linear dependencies among the above $10$ linear forms.  It follows that $\pd(S/L) \le 5$, as $L$ is an unmixed, height three ideal whose generators live in a polynomial ring with $6$ variables.  

If three of these primes have a common linear form, we may assume that either $x = y$ or $x = e$.  In the first case, $L$ contains a linear form and we are done by Lemma~\ref{LinForm}.  In the second case, we still must have $\h(y,g,h,x,a,b)\leq 5$, $\h(y,g,h,x,c,d) \le 5$.  This forces at least two more linear dependencies which, after a linear change of variables, we may take to be one of the following cases:
\begin{itemize}
\item $a=y,\, c=g$
\item $a=g,\, c=h$
\item $a=g,\, c=y + g$
\end{itemize}
One checks that in each of these cases, if there are no remaining linear dependencies among the linear generators, then all quadrics in $L$ are contained in $(x,y)$, contradicting that $L$ contains a height $3$ ideal of quadrics.  Hence there is a fourth linear dependence among the $10$ variables.  It follows again that all quadrics in $L$ can be written in terms of at most $6$ variables, hence $\pd(S/I) \le 6$.\\
\noindent \underline{Case VI. $L_1 = (w,ac,ad,bc,bd)$, $L_2 = (x,y^2,yz,z^2,ey+fz)$}.\\
If $w \in (x,y,z)$, after a linear change of variables we may assume that $w = x$ or $w = y$.  If $w = x$, then $L$ contains a linear form.  If $w = y$, then $x \in (y,a,b,c,d)$ (otherwise all quadrics in $L$ are contained in $(x,z)$, a contradiction), and similarly $z \in (y,a,b,c,d)$. So $L_1 + L_2 = (y,ac,ad,bc,bd,x,z^2,fz)$, whose generators require at most $6$ linear forms to express.  If follows that $\pd(S/(L_1+L_2)) \le 6$, whence $\pd(S/I) \le 6$.

If $w \notin (x,y,z)$ and if $\h(x,y,z,w,a,b) = 6$ or $\h(x,y,z,w,c,d) = 6$, then all quadrics can be expressed in terms of $6$ variables.  Hence we may assume that $\h(x,y,z,w,a,b) \le 5$ and $\h(x,y,z,w,c,d) \le 5$ and without loss of generality that $x = a$ and $y = c$.  Then $L \subseteq (w,xy,xd,by,bd) \cap (x,y,z) \subseteq (wx, wy, wz, xy, xd, by, bd)$ and again all quadrics in $L$ are expressible in terms of at most $6$ variables.\\
\noindent \underline{Case VII. }\\\underline{$L_1 = (s,tv,tw,uv,uw)$, $L_2 = (x,y,z)^2 + (ax+by+cz,dx+ey+fz)$}.\\
If $\h(s,x,y,z) = 4$, then all quadrics in $L$ are in $(ax+by+cz,dx+ey+fz)$, a contradiction on the height of $L$.  Hence we may assume that $s = x$ after a linear change of variables.  As usual we may assume that $ax+by+cz, dx+ey+fz \in L_1$.  Then $L_1 + L_2 = (x,tv,tw,uv,uw,y^2,yz,z^2)$.  If $\h(t,u,v,w,x,y,z) = 7$, then one checks that $\pd(S/(L_1+L_2)) = 6$.  Otherwise, one has $\pd(S/(L_1+L_2)) \le 6$ and hence $\pd(S/I) \le 6$ by Lemmas~\ref{sum} and \ref{ses}.\\
\noindent \underline{Case VIII. $L_1 = (x,y^2,yz,z^2,ay+bz)$, $L_2 = (u,v^2,vw,w^2,cv+dw)$}.\\
After a change of variables, we may suppose $cv + dw \in L_1$, or else any complete intersection of quadrics in $L_1 \cap L_2$ would be expressible in terms of at most $6$ variables.  Hence $L_1 + L_2 \subseteq (x, y^2, yz, z^2, ay+bz, u, v^2, vw, w^2)$.  By symmetry, we can also assume $ay + bz \in (u,v,w)$.  If ${\rm ht}(a,b,y,z,u,v,w) \ge 6$, then we have a contradiction.  Therefore, ${\rm ht}(x,y,z,a,b,u,v,w) \le 6$ and $\pd(S/(L_1 + L_2)) \le 6$.  It follows from Lemmas~\ref{sum} and \ref{ses} that $\pd(S/I) \le 6$.

\noindent \underline{Case IX.}\\\underline{$L_1 = (u,v^2,vw,w^2,gv+hw)$, $L_2 = (x,y,z)^2 + (ax+by+cz,dx+ey+fz)$}.\\
This case is similar to the previous one.  After a change of variables, we may assume that $ax+by+cz,dx+ey+fz \in L_1$, or else all quadrics in $L$ are expressible in terms of at most $6$ linear forms.  Similarly, we may assume that $gv + hw \in (x,y,z)$.  This forces ${\rm ht}(g,h,v,w,x,y,z) \le 5$ and so ${\rm ht}(u,v,w,g,h,x,y,z) \le 6$.  Once again we have $\pd(S/I) \le 6$.  

\noindent \underline{Case X. $L_1 =  (x,y,z)^2 + (ax+by+cz,dx+ey+fz)$,}\\\underline{$L_2 = (u,v,w)^2 + (a'x+b'y+c'z,d'x+e'y+f'z)$}.\\
As before, we may assume $\h(u,v,w,x,y,z) \le 5$ and thus that $x = u$.  We may also assume that $ax+by+cz, dx+ey+fz \in L_2$ and symmetrically that $a'x+b'y+c'z, d'x+e'y+f'z \in L_1$, after a change of variables.  
Suppose $\h(u,v,w,x,y,z) = 5$, then
\[L_1 + L_2 = (x^2, xy, xz, y^2, yz, z^2, xv, xw, v^2, vw, w^2, ax+by+cz, dx+ey+fz).\]
Moreover $(L_1 + L_2):x = L_1:x + L_2:x = (x,y,z,v,w)$ and $(L_1+L_2)+(x) = (x,y^2,yz,z^2,v^2,vw,w^2,by+cz,ey+fz)$.  Clearly $\pd(S/(x,y,z,v,w)) \le 5$.  As $ax+by+cz, dx+ey+fz \in (u,v,w)$, it follows that $b,c,e,f \in (u,v,w,x,y,z)$.  Hence $\pd(L_1+L_2+(x)) \le 5$.  Therefore $\pd(S/(L_1+L_2)) \le 5$ and $\pd(S/I) \le 5$ by Lemmas~\ref{sum} and \ref{ses}.

If $\h(u,v,w,x,y,z) = 4$, we may further suppose that $v = y$.  It follows that we may assume $c = w$ and $f = 0$ after a linear change of coordinates and generators.  Hence $L_1 + L_2 = (x^2,xy,xz,y^2,yz,z^2,xw,yw,w^2,ax+by+wz,dx+ey)$.  If ${\rm ht}(x,y,z,w,d,e) \ge 5$, one checks that $\pd(S/(L_1+L_2)) = 5$; otherwise, $L_1 + L_2$ can be expressed in terms of at most $6$ variables.  In either case, $\pd(S/I) \le 6$. 
\end{proof}

\begin{prop}\label{pd4}
Suppose $e(S/L) = 4$.  Then $\pd(S/I) \le 6$ or any three quadrics in $L$ can be written in terms of at most $6$ variables.
\end{prop}

\begin{proof} There are the following $11$ possibilities for the type of $L$:
\begin{enumerate}
\item $\langle 4;1 \rangle$
\item $\langle 2;2 \rangle$
\item $\langle 1;4 \rangle$
\item $\langle 1,3;1,1 \rangle$
\item $\langle 1,1;1,3 \rangle$
\item $\langle 2,2;1,1 \rangle$
\item $\langle 1,2;2,1 \rangle$
\item $\langle 1,1;2,2 \rangle$
\item $\langle 1,1,2;1,1,1 \rangle$
\item $\langle 1,1,1;1,1,2 \rangle$
\item $\langle 1,1,1,1;1,1,1,1 \rangle$
\end{enumerate}
The last $6$ of these cases are covered in Lemma~\ref{pd4cases}.  The cases $\langle 1,3;1,1\rangle$ and $\langle 1,1;1,3 \rangle$ are covered in Lemmas~\ref{pd1311} and \ref{pd1113}, respectively.  If $L$ is of type $\langle 4;1 \rangle$, it is prime and one of the first $3$ cases of Theorem~\ref{e=4}.  All three are Cohen-Macaulay.  Hence $\pd(S/I) \le 4$ by Lemma~\ref{ses}.  If $L$ is of type $\langle 2;2 \rangle$ or $\langle 1;4 \rangle$, then by Propositions~\ref{1;4} and \ref{st2;2}, either $\pd(S/L) \le 5$, all quadrics in $L$ can be expressed in terms at most $6$ variables, or $L$ contains a linear form. 
\end{proof}

 \section{The Remaining Cases of Multiplicity $4$ and $5$}\label{leftover}

Throughout this section $I = (q_1,q_2,q_3,q_4)$ is an ideal generated by $4$ quadrics.  By the previous sections, we have proved the main result in all cases except when $e(S/I) = 4$ or $5$.  In those cases where we have not already produced the desired bound on $\pd(S/I)$, 
we have that any complete intersection of quadrics in $I$ can be written in terms of at most $6$ variables.  (Different choices of the complete intersection may lead to different sets of 6 variables.)  Indeed after taking a linear combinations of the generators, we may assume that $q_1,q_2,q_3$ is the complete intersection in question and apply Prop.~\ref{mult35} or Prop.~\ref{pd4}.  While it is tempting to conclude that this assumption should imply that all four quadrics are expressible in terms of at most $6$ linear forms, the following example shows this is not true.

\begin{eg} Consider the ideal $I = (ax+y^2, bx + yz, cx + z^2, dx)$ in $S = k[a,b,c,d,x,y,z]$.  Note that $\h(I) = 3$ and $I$ is generated by four quadric forms.  It is easy to check that any complete intersection of quadrics $q_1, q_2, q_3$ in $I$ can be expressed in terms of at most $6$ linear forms ($x$, $y$ and $z$ and the three linear coefficients of $x$), but it requires all $7$ variables to write all four minimal generators of $I$.  
\end{eg}

For the remainder of this section, we assume that any complete intersection of three quadrics is extended from a polynomial ring in at most $6$ variables and $e(S/I) = 4$ or $5$.

\begin{lem}\label{monom}
If $I$ contains a product of two linear forms, then $\pd(S/I)\leq 6$.
\end{lem}

\begin{proof}
Assume $xy\in I$ for linear forms $x,y$ so that $I=(xy, q_1, q_2, q_3)$. Then after possibly replacing the $q_i$ with a linear combination of the $q_j$ and $xy$, we may further assume ${\rm ht}(q_1,q_2,q_3) = 3$ and then, by the above, we may assume $q_1, q_2, q_3 \in S' = k[z_1,\ldots,z_6]$.

Set $C=(q_1,q_2,q_3)S$ and observe that $I:x = (C:x) + (y)$. Then we have the short exact sequence 
$$\mbox{(1)} \quad 0\longrightarrow S/(I:x) \longrightarrow S/I \longrightarrow S/(I + (x))\longrightarrow 0$$
 If $x$ is regular on $S/C$, then the first and last term of (1) are of the form $C + (\ell)$ where $\ell$ is a linear form. Write $C + (\ell)$ as $C_1 +(\ell)$, with $\ell$ regular on $S/C_1$. Then $\pd(S/(C + (\ell))) =1+\pd(S/C_1)$. The ideal $C_1$ need not be a complete intersection but is generated by 3 quadrics, hence by Theorem~\ref{ht2}, we have $\pd(S/C_1)\leq 4$. Then the first and last term of (1) have projective dimension at most $5$, and by Depth Lemma we obtain $\pd(S/I)\leq 5$.
 
We may then assume $x$ and $y$ are not regular on $S/C$. Since $C$ is extended from $S'$ and $x$ and $y$ are not regular on $S/C$, it follows that $x,y$ are both contained in some associated prime of $(q_1, q_2, q_3)S$, all of which are extended from $S'$ by \cite[Exercise 4.7]{AM}.  Since $x,y$ are linear forms we have ${\rm ht}(x,y,z_1,\ldots,z_6) = 6$.  Then $I$ itself is extended from $S'=k[z_1,\ldots,z_6]$ and then $\pd(S/I)\leq 6$.
\end{proof}

\begin{lem}\label{radical}
If $I^{un}$ is radical, then $\pd(S/I)\leq 6$.
\end{lem}

\begin{proof}
Let $q_1,q_2,q_3$ be a regular sequence of quadrics in $I$, by assumption there is a polynomial ring $S'=k[z_1,\ldots,z_6]$ such that $q_1,q_2,q_3\in S'$. Let $C'=(q_1,q_2,q_3)S'$ and $C=C'S$ and observe that all the associated primes of $S/C$ are extended from $S'$. Since $C$ is a complete intersection in $I^{un}$ it follows that ${\rm Ass}(S/I^{un})\subseteq {\rm Ass}(S/C)$ yielding that all associated primes of $S/I^{un}$ are extended from $S'$. Since $I^{un}$ is radical then $I^{un}$ itself is extended from $S'$ and then, by Proposition \ref{iungens}, it follows that $\pd(S/I)\leq 6$.
\end{proof}

\begin{lem}\label{xQ}
If there exist quadrics $Q_1,Q_2,Q_3$ and linear forms $x, y$ such that $I\subseteq (x, Q_1, Q_2, Q_3)$ or $I \subseteq (x,y)^2 + (Q_1, Q_2, Q_3)$, then $\pd(S/I)\leq 6$.
\end{lem}

\begin{proof}
In the first case, some linear combination of the $4$ quadric generators of $I$ is $xy$. The statement now follows by Lemma~\ref{monom}.  In the second case, some linear combination of the $4$ quadric generators of $I$ is contained in $(x,y)^2$.  Since $k = \bar{k}$, every such quadric can be written as the product of two linear forms.  We are again done by Lemma~\ref{monom}.
\end{proof}

\begin{lem}\label{primaryreduction}
Let $\p$ be a linear prime, $H$ a $\p$-primary ideal. If $x$ is a linear form in $\p$ which is not contained in $H$ then $H:x$ and $(H + (x))^{un}$ are $\p$-primary and $e(S/H)=e(S/H:x)+e(S/(H + (x)))$.
\end{lem}

\begin{proof}
Since $\ass(S/(H:x)) \subseteq \ass(S/H)$, it is clear that $H:x$ is $\p$-primary; $(H + (x))^{un}$ is $\p$-primary because $H$ is unmixed, of equal height and contained in it.  The second statement follows by additivity of $e(-)$ applied to the following short exact sequence where all modules have the same dimension:
$$0 \longrightarrow S/(H:x) \longrightarrow S/H \longrightarrow S/(H + (x)) \longrightarrow 0$$
together with the well-known fact that $e(S/(H,x)) = e(S/(H + (x))^{un})$. 
\end{proof}

\begin{prop}\label{e4or5}
If $e(R/I)=4$ or $5$ and every complete intersection of 3 quadrics in $I$ can be written in terms of at most 6 linear forms, then $\pd(R/I)\leq 6$.
\end{prop}

\begin{proof}
Consider the primary components of $I^{un}$.   The hypotheses allow us to consider three cases: when $I^{un}$ is radical or has components of the form $(x,y,z^2)$, when $I^{un}$ has a component primary to a multiplicity $2$ prime that is not itself prime, and when $I^{un}$ has a component primary to a linear prime which is neither prime nor of the form $(x,y,z^2)$.\\
\noindent \underline{Case 1: There is a primary component $J$ of type $\langle 2; 2 \rangle$.}\\
Then $J$ has one of the $4$ forms listed in Proposition~\ref{st2;2}.  All four cases are covered by Lemma~\ref{xQ}.  \\
\noindent \underline{Case 2: There is a primary component $J$ of type $\langle 1; e \rangle$, with $e \ge 2$.}\\
First consider the case when $e = 3$.  Then $J$ has one of the $7$ forms in Proposition~\ref{1;3}.  Cases (1), (2), (4) and (5) are covered by Lemma~\ref{xQ}.  Cases (6) and (7) easily follow as well.  In the remaining case (3), we have $J = (x^2,xy,xz,y^2,yz,ax+by+z^2,cx+dy)$.  By expressing $q_1,\ldots,q_4$ in terms of the minimal generators of $I$ and taking linear combinations we may ensure that $q_4 \in (x,y,z)^2$ and that $q_1, q_2, q_3$ for a complete intersection.  Hence we may assume that $q_1, q_2, q_3$ are extended from a polynomial ring $R$ in at most $6$ variables and since $(x,y,z)$ minimally contains $(q_1, q_2, q_3)$, the prime $(x,y,z)$ is also extended from $R$.  Hence $q_4$ is extended from $R$ and we are done by Lemma~\ref{regseq}.\\
If $e = 5$, then let $\p = (x, y, z)$ be the linear prime for which $J$ is $\p$-primary.  Let $\ell \in \p - J$.  By Lemma~\ref{primaryreduction}, $J:\ell$ and $J + (\ell)$ are $\p$-primary and $e(S/(J:\ell)) + e(S/(J+(\ell))) = 5$.  Either one of these multiplicities is $3$ and we are done by the previous argument, or one of these multiplicities is $4$ and we reduce to the the next case.

If $e = 4$, we consider $J:\p$, which contains $J$ and is $\p$-primary with strictly smaller multiplicity.  If $J:\p$ contains a linear form $\ell$, then $J:\ell = \p$.  Hence $e(S/(J + (\ell))) = 3$ and we are again done by a previous case.  If $J:\p$ does not contain a linear form, then either $e(S/(J:\p)) = 3$ or all quadrics in $J:\p$, and hence $I$, can be written in terms of $6$ variables, in which case we are done again.  Otherwise by Proposition~\ref{1;2}, $J:\p = (x,y,z)^2 + (ax+by+cz, dx+ey+fz)$, where $a,\ldots,f \in S_1$.  For any ideal $I$ generated by $4$ quadrics contained in this ideal, either the generators are expressible in terms of $6$ linear forms, or without loss of generality we can take  $q_1 = ax + by + cz+q_1'$, $q_2 = dx + ey + fz+q2'$, where $q_1', q_2', q_3, q_4 \in (x,y,z)^2$.  Since $(x,y,z) \in \ass(S/I) \subseteq \ass(S/(q_1, q_2, q_3))$, and since $q_1, q_2, q_3$ are extended from a $6$ variable polynomial ring, we may take $x, y, z$ to be $3$ of these variables.  It follows that all the generators of $I$ can be written in terms of at most $6$ variables and so $\pd(S/I) \le 6$.

If $e = 2$, then $J$ is of one of the $4$ forms in Proposition~\ref{1;2}.  The previous arguments cover cases (3) and (4).  In case (2), $J = (x, y^2, yz, z^2, ay + bz)$ with ${\rm ht}(x,y,z,a,b) = 5$.  Either $I$ contains a monomial or else we can write $I = (q_1, q_2, q_3, q_4)$ with $q_1 = c_1x + y^2$, $q_2 = c_2x + yz$, $q_3 = c_3x + z^2$ and $q_4 = c_4x + ay + bz$.  If ${\rm ht}(c_1, c_2, c_3, c_4, x, y, z, a, b) = 5$, we are done.  Otherwise, after possibly taking a linear combination of the generators we may assume ${\rm ht}(c_4, x, y, z, a, b) = 6$ and $(q_1, q_2, q_4)$ and $(q_1, q_3, q_4)$ form a compete intersections.  It follows that all $4$ quadrics are extended from $k[c_4, x, y, z, a, b]$ and hence $\pd(S/I) \le 6$.\\
\noindent \underline{Case 3: Every primary component of $I^{un}$ is prime or is equal to $(x, y, z^2)$.}\\
  Since $\ass(S/I^{un}) \subseteq \ass(S/(q_1, q_2, q_3))$, each associated prime of $I^{un}$ is extended from a single polynomial ring $R'$ in at most $6$ variables by Proposition~\ref{iungens}.  In the case of a component of the form $(x, y, z^2)$, the prime ideal $(x,y,z)$ is extended from $R'$ and hence $(x, y, z^2)$ is as well.  Since all primary components of $I^{un}$ are extended from $R'$ it follows from that $\pd(S/I) \le 6$.
\end{proof}

\section{Main Result}\label{main}

\begin{thm}\label{pd6}
Let $S$ by a polynomial ring over a field $K$, and let $I = (q_1,q_2,q_3,q_4)$ be an ideal of $S$ minimally generated by $4$ homogeneous polynomials of degree $2$.
Then ${\rm pd}(S/I)\leq 6$.  Moreover, this bound is tight.
\end{thm}

\begin{proof} By extension of scalars, we may assume that $K$ is algebraically closed.  As previously noted, $\h(I) \le 4$.  If $\h(I) = 1$ or $4$, then $\pd(S/I) = 4$.  If $\h(I) = 2$, then $\pd(S/I) \le 6$ by Theorem~\ref{ht2}.  So we may assume that $\h(I) = 3$.

By Theorem~\ref{e6}, we have $e(S/I) \le 6$.  Moreover, if $e(S/I) = 6$, then $\pd(S/I) = 3$.  The cases where $e(S/I) = 1,2,3,4,5$ are proved in Propositions~\ref{e=1}, \ref{mult2}, \ref{mult35}, \ref{pd4} and \ref{e4or5}.  
\end{proof}

\begin{table}
\begin{tabular}{|l|l|l|l|}
\hline
{$\mathbf{e}(S/I)$} & 
$\langle\underline{e};\underline{\lambda}\rangle$ for $I^{un}$& Bound for $\pd(S/I)$ & Justification\\
\hline\hline
{\bf 1}&$\langle1;1\rangle$&4& Proposition \ref{e=1}\\
\hline
{\bf 2}&$\langle2;1\rangle$& 4 & Lemma \ref{iuncm} \\
&$\langle1;2\rangle$ & 5 & Proposition \ref{mult2}\\
&$\langle1,1;1,1\rangle$& 5 & Proposition \ref{mult2}\\
\hline
{\bf 3} & $\langle3;1\rangle$ & 5 &  Lemma \ref{LinForm}\\
&$\langle1;3\rangle$ & 6 & Lemma \ref{pd13}\\
&$\langle1,2;1,1\rangle$ & 6 & Lemma \ref{pd1211}\\
&$\langle1,1;1,2\rangle$ & 6 &  Lemma \ref{pd1112}\\
&$\langle1,1,1;1,1,1\rangle$ & 6 & Lemma \ref{pd111111}\\
\hline
{\bf 6}& {\rm any}& 3 & Theorem \ref{e6}\\
\hline \hline
\hline
{$\mathbf{e}(S/L)$} & 
$\langle\underline{e};\underline{\lambda}\rangle$ for $L$& Bound for $\pd(S/I)$ & Justification\\
\hline\hline
{\bf 4}&$\langle4;1\rangle$ & 4 & Lemma  \ref{ses}\\ 
&$\langle2;2\rangle$ & 5 &Proposition \ref{st2;2}* \\
&$\langle1;4\rangle$ & 6 & Proposition \ref{1;4}* \\
&$\langle1,3;1,1\rangle$ & 5& Proposition \ref{pd1311} \\
&$\langle1,1;1,3\rangle$ & 6 & Proposition \ref{pd1113}* \\
&$\langle2,2;1,1\rangle$ &  6 & Proposition \ref{pd4cases}* \\
&$\langle1,2;2,1\rangle$ & 6& Proposition \ref{pd4cases}* \\
&$\langle1,1;2,2\rangle$ & 6&Proposition \ref{pd4cases}* \\
&$\langle1,1,2;1,1,1\rangle$ & 6 &Proposition \ref{pd4cases}* \\
&$\langle1,1,1;1,1,2\rangle$ & 6 & Proposition \ref{pd4cases}* \\
&$\langle1,1,1,1;1,1,1,1\rangle$ & 6  & Proposition \ref{pd4cases}*\\
\hline
{\bf 3} & $\langle3;1\rangle$ & 5 &  Lemma \ref{LinForm}\\
&$\langle1;3\rangle$ & 6 & Lemma \ref{pd13}*\\
&$\langle1,2;1,1\rangle$ & 5 & Lemma \ref{pd1211} \\
&$\langle1,1;1,2\rangle$ & 6 & Lemma \ref{pd1112}* \\
&$\langle1,1,1;1,1,1\rangle$ & 6 & Lemma \ref{pd111111}*\\
\hline
\end{tabular}

\caption{Bounds on $\pd(S/I)$ where $I$ is a height $3$ ideal generated by $4$ quadrics}
\label{table2}
\end{table}

A summary of the many cases in the proof of Theorem~\ref{pd6} is given in Table~\ref{table2}.  The Lemma or Proposition covering each case is referenced on the right.  Asterisks denote cases that are also covered by Proposition~\ref{e4or5}.

We remark that there are canonical examples of ideals $I_2$ and $I_3$ generated by $4$ quadrics of heights $2$ and $3$, respectively, and with $\pd(S/I_2) = \pd(S/I_3) = 6$.  Namely, we can take $I_2 = (x^2, y^2, ax+by, cx+dy)$ and $I_3 = (x^2, y^2, z^2, ax+by+cz)$ in the polynomial ring $S = k[a,b,c,d,x,y,z]$.  These correspond to the examples $I_{2,2,2}$ and $I_{3,1,2}$ from \cite{Mc}.   More generally, we previously posed the following question:

\begin{qst}[{ \cite[Question 6.2]{HMMS}}]\label{qst2}
Let $S$ be a polynomial ring and let $I$ be an ideal of $S$ generated by $n$ quadrics and having  $\h I=h$. Is it true that $\pd (S/I)\leq h(n-h+1)$?
\end{qst}

There are examples in \cite{Mc} achieving this bound for all possible integers $h$ and $n$.  This paper now answers this question affirmatively if $n \le 4$, while our previous paper \cite{HMMS} gave an affirmative answer if $h \le 2$.  Note that if the question has a positive answer, it would give a bound on the projective dimension of ideals generated by $n$ quadrics that is quadratic in $n$ -- much smaller than the known bounds of Ananyan-Hochster.

\appendix
\section{Resolutions of Primary Ideals}\label{append}

In this section we collect the details of the many unmixed and primary ideals listed in the structure theorems in Sections~\ref{primary} and \ref{22struct}.  The technique is the same for each of the following: one resolves the given ideal generically and uses the Buchsbaum-Eisenbud exactness criteria to check the conditions that ensure the resolution is exact.  If $F_\*$ is the resolution of $S/I$ and $\partial_i$ denotes the $i^{\text th}$ differential map, then this amounts to checking that $\h(I_{r_j}(\partial_j)) \ge j$ for all $j$.  Here $r_j = \sum_{i = j}^{p} (-1)^{p-i} \rank(F_i)$ denotes the expected rank of the $j$th map, and $I_r(\partial_j)$ denotes the ideal of $r \times r$ minors of the matrix associated to the $j$th differential.  Moreover one checks that $\h(I_{r_j}(\partial_j)) \ge j+1$ for $j > \h(I)$ to ensure that the ideal is unmixed (cf. \cite[Proposition 2.4]{HMMS2}).  In the case that $\h(I) = 3$, this amounts to showing that 
\[\h(I_{r_j}(\partial_j)) \ge \begin{cases} j & \text{ if } j = 1,2,3\\
j+1 & \text{ if } j \ge 4\end{cases}.\]
We do this explicitly for Lemmas~\ref{primary6} and \ref{primary21}.  The reader may check the remaining ones with the help of a computer algebra system, such as Macaulay2 \cite{Mac2}.  We have posted a supplementary Macaulay2 \cite{Mac2} file at \verb+http://www.math.unl.edu/~aseceleanu2/research/fourQuadrics.m2+

Throughout this section $x,y,z$ represent independent linear forms so that $\p = (x,y,z)$ is a height three linear prime ideal.  

\begin{lem}\label{primary6}
If
\[
L = (x^2,y^2,z^2,xyz,Cxy+Bxz+Ayz),\]
where $\h(x,y,z,A,B,C) \ge 5$, then $L$ is $(x,y,z)$-primary, $\pd(S/L) = 4$ and $e(S/L) = 6$.
\end{lem}

\begin{proof}   Consider the complex
      \[\xymatrix{
S & \ar[l]_{\partial_1} S^5 & \ar[l]_{\partial_2} S^{9} & \ar[l]_{\partial_3} S^6 & \ar[l]_{\partial_4} S^1 & \ar[l] 0,}      \]
where
$\partial_1 = { \begin{pmatrix}x^{2}&
      y^{2}&
      z^{2}&
      C x y+B x z+A y z&
      x y z\\
      \end{pmatrix}}$\\
$\partial_2 = { \begin{pmatrix}-C y-B z&
      0&
      {-y^{2}}&
      0&
      {-y z}&
      0&
      {-z^{2}}&
      0&
      0\\
      0&
      -C x-A z&
      x^{2}&
      0&
      0&
      {-x z}&
      0&
      0&
      {-z^{2}}\\
      0&
      0&
      0&
      B x+A y&
      0&
      0&
      x^{2}&
      {-x y}&
      y^{2}\\
      x&
      y&
      0&
      {-z}&
      0&
      0&
      0&
      0&
      0\\
      {-A}&
      {-B}&
      0&
      C&
      x&
      y&
      0&
      z&
      0\\
      \end{pmatrix}}$\\
$\partial_3 = { \begin{pmatrix}{-y}&
      {-z}&
      0&
      0&
      0&
      0\\
      x&
      0&
      {-z}&
      0&
      0&
      0\\
      C&
      0&
      0&
      z&
      0&
      0\\
      0&
      {-x}&
      {-y}&
      0&
      0&
      0\\
      B&
      C&
      0&
      {-y}&
      {-z}&
      0\\
      {-A}&
      0&
      C&
      x&
      0&
      {-z}\\
      0&
      B&
      0&
      0&
      y&
      0\\
      0&
      {-A}&
      {-B}&
      0&
      x&
      y\\
      0&
      0&
      A&
      0&
      0&
      x\\
      \end{pmatrix}}$\\
$\partial_4 = { \begin{pmatrix}{-z}\\
      y\\
      {-x}\\
      C\\
      {-B}\\
      A\\
      \end{pmatrix}}$\\
      Note that $x^2 \in I_1(\partial_1)$, $x^6, y^6 \in I_4(\partial_2)$ and $x^5,y^5,z^5 \in I_5(\partial_3)$ and $I_1(\partial_4) = (x,y,z,A,B,C)$.  It follows that this is exact and forms a resolution of $L$, $\pd(S/L) = 4$ and $L$ is unmixed.  As $\sqrt{L} = \p$ and $\lambda(S_\p/L_\p) = 6$, it follows that $L$ is $(x,y,z)$-primary and $e(S/L) = 6$.
\end{proof}

\begin{lem}\label{primary21}
If 
\[J = (x,y,z)^2 + (ax+by+cz,dx+ey+fz),\]
where $\h\left((x,y,z)+I_2\begin{pmatrix}a&b&c\\d&e&f\end{pmatrix}\right) = 5$, then $J$ is $(x,y,z)$-primary, $\pd(S/J) = 4$ and $e(S/J) = 2$.  Moreover, 
\begin{align*}
L &= (x^2,y^2,z^2):J \\
&= (x^2,y^2,z^2,xyz, (ae-bd)xy + (af-cd)xz + (bf-ce)yz)\\
\end{align*}
and $\pd(S/L) = 4$.
\end{lem}

\begin{proof}
Consider the complex
\[\xymatrix{
S & \ar[l]_{\partial_1} S^8 & \ar[l]_{\partial_2} S^{14} & \ar[l]_{\partial_3} S^9 & \ar[l]_{\partial_4} S^2 & \ar[l] 0,}      \]
where\\
$\partial_1 = { \begin{pmatrix}a x+b y+c z&
      d x+e y+f z&
      x^{2}&
      x y&
      y^{2}&
      x z&
      y z&
      z^{2}\\
      \end{pmatrix}}$\\
      $\partial_2 = { \begin{pmatrix}{-x}&
      0&
      {-y}&
      0&
      0&
      0&
      {-z}&
      0&
      0&
      0&
      0&
      0&
      0&
      0\\
      0&
      {-x}&
      0&
      {-y}&
      0&
      0&
      0&
      {-z}&
      0&
      0&
      0&
      0&
      0&
      0\\
      a&
      d&
      0&
      0&
      {-y}&
      0&
      0&
      0&
      {-z}&
      0&
      0&
      0&
      0&
      0\\
      b&
      e&
      a&
      d&
      x&
      {-y}&
      0&
      0&
      0&
      {-z}&
      0&
      0&
      0&
      0\\
      0&
      0&
      b&
      e&
      0&
      x&
      0&
      0&
      0&
      0&
      0&
      {-z}&
      0&
      0\\
      c&
      f&
      0&
      0&
      0&
      0&
      a&
      d&
      x&
      y&
      {-y}&
      0&
      {-z}&
      0\\
      0&
      0&
      c&
      f&
      0&
      0&
      b&
      e&
      0&
      0&
      x&
      y&
      0&
      {-z}\\
      0&
      0&
      0&
      0&
      0&
      0&
      c&
      f&
      0&
      0&
      0&
      0&
      x&
      y\\
      \end{pmatrix}}$\\
      $\partial_3 = { \begin{pmatrix}y&
      0&
      z&
      0&
      0&
      0&
      0&
      0&
      0\\
      0&
      y&
      0&
      z&
      0&
      0&
      0&
      0&
      0\\
      {-x}&
      0&
      0&
      0&
      z&
      0&
      0&
      0&
      0\\
      0&
      {-x}&
      0&
      0&
      0&
      z&
      0&
      0&
      0\\
      a&
      d&
      0&
      0&
      0&
      0&
      z&
      0&
      0\\
      b&
      e&
      0&
      0&
      0&
      0&
      0&
      z&
      0\\
      0&
      0&
      {-x}&
      0&
      {-y}&
      0&
      0&
      0&
      0\\
      0&
      0&
      0&
      {-x}&
      0&
      {-y}&
      0&
      0&
      0\\
      0&
      0&
      a&
      d&
      0&
      0&
      {-y}&
      0&
      0\\
      0&
      0&
      b&
      e&
      a&
      d&
      x&
      {-y}&
      0\\
      c&
      f&
      b&
      e&
      0&
      0&
      0&
      {-y}&
      z\\
      0&
      0&
      0&
      0&
      b&
      e&
      0&
      x&
      0\\
      0&
      0&
      c&
      f&
      0&
      0&
      0&
      0&
      {-y}\\
      0&
      0&
      0&
      0&
      c&
      f&
      0&
      0&
      x\\
      \end{pmatrix}}$\\
      $\partial_4 = { \begin{pmatrix}{-z}&
      0\\
      0&
      {-z}\\
      y&
      0\\
      0&
      y\\
      {-x}&
      0\\
      0&
      {-x}\\
      a&
      d\\
      b&
      e\\
      c&
      f\\
      \end{pmatrix}}$\\
      It is easy to check that this is a complex.  We note that $x^2 \in I_1(\partial_1)$, $x^7,y^7 \in I_7(\partial_2)$, $x^7,y^7,z^7 \in I_7(\partial_3)$ and $x^2,y^2,z^2,ae-bd,af-cd,bf-ce \in I_2(\partial_4)$.  It follows that the complex is exact and resolves $J$.  Moreover, we have that $\pd(S/J) = 4$ and $J$ is unmixed.  As $\sqrt{J} = (x,y,z) = \p$ and $\lambda(S_\p/J_\p) = 2$, it follows that $J$ is $(x,y,z)$-primary and $e(S/J) = 2$.  
     The inclusion $L \subseteq (x^2,y^2,z^2):I$ is clear.  Moreover $L$ is $(x,y,z)$-primary by Lemma~\ref{primary6}. Since both ideals are height $3$ unmixed of multiplicity $6$, they must be equal by Lemma~\ref{unmixedequal}.
\end{proof}

\begin{lem}\label{primary22}
If 
\[J = (x,y,z)^2 + (bx-ay,cx-az,cy-bz),\]
where $\h(x,y,z,a,b,c) = 6$, then $J$ is $(x,y,z)$-primary, $\pd(S/J) = 5$ and $e(S/J) = 2$.  %Moreover, 
%\[L = (x^2,y^2,z^2):J = (x^2,y^2,z^2,xyz, cxy + bxz + ayz)\]
%and $\pd(S/L) = 4$.
\end{lem}

\begin{lem}\label{primary31}
If 
\[J = (x,y,z)^2 + (ax+by+cz),\]
where $\h(x,y,z,a,b,c) \ge 5$, then $J$ is $(x,y,z)$-primary, $\pd(S/J) = 4$ and $e(S/J) = 3$.  
\end{lem}

\begin{lem}\label{primary32}
If 
\[J = (x,y^2,yz,z^3,ay+z^2),\]
then $J$ is $(x,y,z)$-primary, $\pd(S/J) = 3$ and $e(S/J) = 3$.  
\end{lem}

\begin{lem}\label{primary33}
If 
\[J = (x^2,xy,xz,y^2,yz,ax+by+z^2,cx+dy),\]
where $\h(x,y,z,c,d) = 5$, then $J$ is $(x,y,z)$-primary, $\pd(S/J) = 4$ and $e(S/J) = 3$.  Moreover, 
\[L = (x^2,y^2,ax+by+z^2):J = (x^2,xy,y^2,ax+by+z^2,cxz - dyz)\]
and $\pd(S/L) = 4$.
\end{lem}

\begin{lem}\label{primary34}
\[J = (x,y^3,y^2z,yz^2,z^3,ay + bz),\]
where $\h(x,y,z,a,b) = 5$, then $J$ is $(x,y,z)$-primary, $\pd(S/J) = 4$ and $e(S/J) = 4$.
\end{lem}

\begin{lem}\label{primary35}
If
\[J = (x^2,xy,xz,yz^2,z^3,ax+by+cz,dx+y^2),\]
where $\h(x,y,z,b,c) = \h(x,y,z,c,d) = 5$, then $J$ is $(x,y,z)$-primary, $\pd(S/J) = 4$ and $e(S/J) = 3$.
\end{lem}

\begin{lem}\label{primary36}
If
\[J = (x^2,xy,xz,y^3,z^3,ax+by+cz,dx+yz),\]
where $\h(x,y,z,b,c) = \h(x,y,z,b,d) = \h(x,y,z,c,d) = 5$, then $J$ is $(x,y,z)$-primary, $\pd(S/J) = 4$ and $e(S/J) = 3$.
\end{lem}

\begin{lem}\label{primary37}
If
\[J = (x^2,xy,xz,z^3,c x+y^{2},b x-y z,a x+b y+c z),\]
where $\h(x,y,z,b,c) = 5$, then $J$ is $(x,y,z)$-primary, $\pd(S/J) = 4$ and $e(S/J) = 3$.
\end{lem}

\begin{lem}\label{primary38}
If
\[J = (x,y,z)^3 +(ax+by+cz,dx+ey+fz),\]
where $\h\left((x,y,z)+I_2\begin{pmatrix}a&b&c\\d&e&f\end{pmatrix}\right) = 5$ and $\h(x,y,z,a,b,c,d,e,f) \ge 6$, then $J$ is $(x,y,z)$-primary, $\pd(S/J) = 5$ and $e(S/J) = 3$.  
\end{lem}

\begin{lem}\label{primary310}
If
\[J = (x,y,z)^3 + (ax+by + q,ay+bz + q',xz-y^2),\]
where $\h(x,y,z,a,b) = 5$ and $q,q' \in (x,y,z)^2$, then $J$ is $(x,y,z)$-primary, $\pd(S/J) = 4$ and $e(S/J) = 3$.
\end{lem}

\begin{lem}\label{primary42} If
\[J = (x^2,xy,xz,y^2,yz,z^3,ax+by+z^2),\]
where $\h(x,y,z) = 3$, then $J$ is $(x,y,z)$-primary, $\pd(S/J) = 3$ and $e(S/J) = 4$.
\end{lem}

\begin{lem}\label{primary43} If
\[J = (x^2, xy, xz, y^3, y^2z, yz^2, z^3, ax+by+cz),\]
where $\h(x,y,z,b,c) = 5$, then $J$ is $(x,y,z)$-primary, $\pd(S/J) = 4$ and $e(S/J) = 4$.
\end{lem}

\begin{lem}\label{primary44} If
\[J = (x^2, xy, y^2, ax + by + z^2, cx + dy) ,\]
where $\h(x,y,z,c,d) = 5$, then $J$ is $(x,y,z)$-primary, $\pd(S/J) = 4$ and $e(S/J) = 4$.
\end{lem}

This lemma gives the only nontrivial $\langle2;2\rangle$ structure we need to consider.

\begin{lem}\label{22gen} If 
\[J = (x^2, xy, y^2, ax+by, cx+dy, ad-bc+ex+fy),\]
where $(x,y,ad-bc)$ is a height $3$ prime ideal, then $J$ is $(x,y,ad-bc)$-primary, $\pd(S/J) = 3$ and $e(S/J) = 4$.
\end{lem}

The last two lemmas are referenced in the body of the paper.

\begin{lem}\label{1211case} If
\[J = (ax+by+cz) + (x,y,z)(u,v),\]
where $\h(u,v,x,y,z) = 5$ and $\h(u,v,a,b,c) \ge 3$, then $J$ is unmixed and has resolution
  \[\xymatrix{
S & \ar[l]_{\partial_1} S^7 & \ar[l]_{\partial_2} S^{11} & \ar[l]_{\partial_3} S^6 & \ar[l]_{\partial_4 = {\small \begin{pmatrix} 0\\-z\\y\\-x\\-v\\u\end{pmatrix}}} S^1 & \ar[l] 0,}      \]
\end{lem}

\begin{lem}\label{1112case} If
\[J = (x^2,xy,xz,y^2,yz,ax+by+wz,cx+dy),\]
where $\h(x,y,w) = 3$ and $\h(c,d,x,y,z) = 5$, then $J$ is unmixed and has resolution
  \[\xymatrix{
S & \ar[l]_{\partial_1} S^7 & \ar[l]_{\partial_2} S^{11} & \ar[l]_{\partial_3} S^6 & \ar[l]_{\partial_4 = {\small \begin{pmatrix} 0\\-z\\y\\-x\\c\\d\end{pmatrix}}} S^1 & \ar[l] 0,}      \]
\end{lem}

\section*{Acknowledgements}
This work was started while all the authors were members of the MSRI program on Commutative Algebra. We thank MSRI for its support and great mathematical environment.  We thank the referee for several suggestions that improved this paper.   We also thank Chuck Weibel for  handling our paper through the later editorial stages.

\bibliographystyle{amsplain}
\bibliography{bibht3}

\end{document}